\documentclass[11pt]{amsart}
\usepackage{fullpage}
\usepackage[latin2]{inputenc}
\usepackage[english]{babel} 
\usepackage{tikz}
\usepackage{amsmath}
\usepackage{amsthm}
\usepackage{amssymb}
\usepackage{cite}
\usepackage{amsrefs}
\usepackage{rotating}
\usepackage{graphicx}
\usepackage{subfigure}

\usetikzlibrary{decorations.pathreplacing}

\newtheorem{thm}{Theorem}[section]
\newtheorem{prop}[thm]{Proposition}

\newtheorem{conj}[thm]{Conjecture}

\newtheorem{lem}[thm]{Lemma}
\newtheorem{defi}[thm]{Definition}

\newtheorem{question}[thm]{Question}
\newtheorem{claim}[thm]{Claim}

\newtheorem{corr}[thm]{Corollary}

\theoremstyle{remark}
\newtheorem*{remark}{Remark}

\begin{document}

\title{On $k$-neighbor separated permutations}

\makeatother
\author{István Kovács}
\address{Department of Control Engineering and Information Technology \\ Budapest University of Technology and Economics}
\email[István Kovács]{kovika91@gmail.com}\thanks{\noindent The research of the first author was
supported by National Research, Development and Innovation Office NKFIH, K-111827.}

\author{Daniel Soltész}
\address{Alfréd Rényi Institute of Mathematics, Hungarian Academy of Sciences}
\email[Daniel Soltész]{solteszd@math.bme.hu}
\thanks{\noindent The research of the second author was
supported by the Hungarian Foundation for Scientific Research Grant (OTKA) No. 108947 and by the National Research, Development and Innovation Office NKFIH,  No. K-120706.}


\begin{abstract}

Two permutations of $[n]=\{1,2 \ldots n\}$ are \textit{$k$-neighbor separated} if there are two elements that are neighbors in one of the permutations and that are separated by exactly $k-2$ other elements in the other permutation. Let the maximal number of pairwise $k$-neighbor separated permutations of $[n]$ be denoted by $P(n,k)$. In a previous paper, the authors have determined $P(n,3)$ for every $n$, answering a question of Körner, Messuti and Simonyi affirmatively. In this paper we prove that for every fixed positive integer $\ell $, $$P(n,2^\ell+1) = 2^{n-o(n)}. $$
We conjecture that for every fixed even $k$, $P(n,k)=2^{n-o(n)}$. We also show that this conjecture is asymptotically true in the following sense $$\lim_{k \rightarrow \infty} \lim_{n \rightarrow \infty} \sqrt[n]{P(n,k)}=2.$$
Finally, we show that for even $n$, $P(n,n)= 3n/2$.

\end{abstract}

\maketitle
\section{Introduction}
There are numerous results concerning the maximum size of a family of permutations pairwise satisfying some prescribed relation, see \cites{reversefree,cibulka,frankldeza,tintersect,intsurvey,colliding,kms,additive}. There is a natural correspondence between permutations of $n$ elements and Hamiltonian paths in the complete graph $K_n$. Let $G_1$ and $G_2$ be two graphs on the same vertex set, we say that $G_3$ is their union if $V(G_3)=V(G_1)=V(G_2)$ and $E(G_3)= E(G_1) \cup E(G_2).$ Körner, Messuti and Simonyi made the following observation.

\begin{prop}[\cite{komesi}]  \label{prop}
The maximal number of Hamiltonian paths in the complete graph such that every pairwise union contains an odd cycle is equal to the number of balanced bipartitions of the vertex set.  That is, on $2n+1$ vertices their maximal number is $\binom{2n+1}{n}$ and on $2n$ vertices it is $\frac{1}{2}\binom{2n}{n}=\binom{2n-1}{n}$.
\end{prop}

The upper bound follows by observing that a Hamiltonian path is a bipartite graph with a balanced bipartition. The union of two paths with the same bipartition is a bipartite graph which clearly cannot contain any odd cycle. On the other hand, if we choose a unique Hamiltonian path for every balanced bipartition, the resulting family satisfies our condition. In \cite{komesi} the authors asked whether the answer remains the same if we ask for a triangle instead of an odd cycle. This question was answered affirmatively.

\begin{thm}[\cite{original}] \label{previousmain}
The maximum number of Hamiltonian paths in the complete graph $K_n$ such that every pairwise union contains a triangle is equal to the number of balanced bipartitions of $[n]$.
\end{thm}

There are two natural ways to generalize this problem. We can consider it a problem for Hamiltonian paths and ask for a $k$-cycle instead of a triangle.  In this case, when $k$ is even there are constructions of size larger than exponential, see \cite{k=4}. In this paper we take a different approach and we formulate the problem in the language of permutations.

\begin{defi}
We say that two permutations of $[n]=\{1,2 \ldots n\}$ are \textit{$k$-neighbor separated} if there are two elements that are neighbors in one permutations and in the other permutation they are separated by exactly $k-2$ elements. Let the maximal number of pairwise $k$-neighbor separated permutations of $[n]$ be denoted by $P(n,k)$ and let  $P(k):= \lim_{n \rightarrow \infty} \sqrt[n]{P(n,k)}.$
\end{defi}


Two permutations are $2$-neighbor different if and only if their corresponding Hamiltonian paths share an edge. Determining $P(n,2)$ is a significantly different task than determining $P(n,k)$, for any $k>2$. The reason for this is that it is that when $k=2$ we are looking for similar Hamiltonian paths instead of different ones. Therefore the problem resembles Erdős-Ko-Rado like intersection problems. We will show that $P(n,2)$ is the number of Hamiltonian paths containing a fixed edge. For the upper bound we will use Katona's cycle method. We postpone this proof to Section \ref{sec:katona}.

Two permutations are $3$-neighbor separated if and only if their corresponding Hamiltonian paths form a triangle in their union. Note that in the case of $P(n,3)$, the upper bound is immediate and the construction is non trivial. Making it significantly different from the case of $P(n,2)$. The statement of Theorem \ref{previousmain} is equivalent to saying \[P(n,3)= \begin{cases} \binom{n}{\left \lfloor \frac{n}{2} \right \rfloor} & \text{ when } n \equiv 1 \mod{2} \\
\frac{1}{2}\binom{n}{\left \lfloor \frac{n}{2} \right \rfloor} & \text{ when } n \equiv 0 \mod{2}.
\end{cases}\]

\noindent  From Theorem \ref{previousmain} it also follows that $P(3)=2$. For $k>3$ the property that two permutations are $k$-neighbor separated is stronger than the requirement that the corresponding Hamiltonian paths form a $k$-cycle. The authors conjecture the following.

\begin{conj}\label{conj:main}
For every integer $3<k , we have P(k) \geq 2$.
\end{conj}

Note that from Conjecture \ref{conj:main} it follows that for all positive integers $k$,  $P(2k+1)=2$ by the fact that the number of pairwise $2k$-separated permutations is at most the number of $C_{2k+1}$-different Hamiltonian paths. But the number of these Hamiltonian paths is at most the number of balanced bipartitions of $[n]$ by the upper bound in Proposition \ref{prop}. For even values of $k$, the authors are more cautious, as in these cases sometimes there are larger constructions (found by computer) than the number of balanced bipartitions of the ground set, although these constructions are still smaller than $2^n$. The main result of the present paper is that Conjecture \ref{conj:main} holds for infinitely many $k$.

\begin{thm}\label{thm:main}
For every positive integer $\ell$, $P(2^\ell+1) = 2$.
\end{thm}

\begin{corr}
For every positive integer $\ell$, the maximal number of Hamiltonian paths of $K_n$ where every pairwise union contains a $C_{2^\ell+1}$ is $2^{n-o(n)}.$
\end{corr}
\begin{proof}
The construction follows from Theorem \ref{thm:main} and for the upper bound, the proof of the upper bound of Proposition \ref{prop} applies verbatim.
\end{proof}

\noindent We also prove upper bounds to $P(n,k)$ for every $k$.

\begin{thm} \label{thm:upper}
For fixed $2<k$ we have \[P(n,k) \leq \begin{cases} 2^{n} & \text{when } k \equiv 1 \mod{2}\\
2^{H\left( \frac{k-1}{2k},\frac{k-1}{2k},\frac{2}{2k}\right)n} & \text{when } k \equiv 0 \mod{2}
\end{cases} \]
where $H(x,y,z)$ is the entropy function.
\end{thm}\noindent From Theorem \ref{thm:upper} it follows that $\limsup_{k \rightarrow \infty} P(k)=2$. We also prove that $\liminf_{k \rightarrow \infty} P(k)=2$ which results in $\lim_{k \rightarrow \infty} P(k)=2.$ We also investigate $P(n,n)$.

\begin{thm} \label{k=n} For every positive integer $n$ the following holds
\[\def\arraystretch{1.3}
\begin{array}{rrll}
& P(n,n)= & \frac{3}{2}n &  \text{when } n \equiv 0 \mod{2}\\
\left\lfloor \frac{3}{2}n \right\rfloor  -1 & \leq P(n,n) \leq & \left\lfloor \frac{3}{2}n \right\rfloor & \text{when } n \equiv 1 \mod{2}.
\end{array}
\]
\end{thm}
\noindent A quick investigation by computer shows that the value of $P(3,3)$ and $P(5,5)$ is equal to the corresponding lower bound of Theorem \ref{k=n}, but the value of $P(7,7)$ attains the upper bound. 

The paper is organized as follows. In Section \ref{sec:lower_bound} we prove Theorem \ref{thm:main}. In Section \ref{sec:classic} we elaborate on the connection between the problem of determining $P(n,k)$ and some Bollobás-type questions. In Section \ref{sec:otherk} we prove that $\lim_{k \rightarrow \infty}P(k)=2$. In Section \ref{sec:pnn} we prove Theorem \ref{k=n}, followed by concluding remarks and some open questions.

\section{When $k=2$}\label{sec:katona}

In this short section we determine the exact value of $P(n,2)$ for all $n$. We use the following equivalent formulation: Two permutations are $2$-neighbor different if and only if their corresponding Hamiltonian paths share an edge. For the upper bounds we will use the following claims.

\begin{claim} \label{claim:hamdecomp}
When $n$ is even, the edges of the complete graph $K_n$ can be decomposed into edge disjoint Hamiltonian paths. 
\end{claim}
\begin{proof}
Let us refer to the vertices of $K_n$ as $\{1,2,\ldots, n\}$. Let $\mathcal{F}$ consist of the Hamiltonian path $\{1,n,2,n-1,\ldots, n/2\}$  plus its first $n/2-1$ rotated versions. One can see that the $n/2$ Hamiltonian paths in $\mathcal{F}$ are pairwise edge disjoint by arranging the vertices $\{1,2,\ldots, n\}$  into a regular $n$-gon. If two Hamiltonian path shares an edge that edge must have the same (euclidean) length in both paths. But edges of the same length come in antipodal pairs in the Hamiltonian paths of $\mathcal{F}$ (except the single edge of maximal length in each path). Therefore rotation by at most $n/2-1$ produces pairwise edge disjoint paths.
\end{proof}

\begin{claim} \label{claim:hamcover}
When $n$ is odd, there is a set $\mathcal{F}$ of Hamiltonian paths of $K_n$ with size $|\mathcal{F}|=n$. And the intersection graph $G$ of $\mathcal{F}$ (the vertices are the Hamiltonian paths, two vertices are adjacent if the paths share an edge) is a cycle of length $n$. 
\end{claim}
\begin{proof}
Let us refer to the vertices of $K_n$ as $\{1,2,\ldots, n\}$. Let  $M_1$ be  the matching $\{(2,n),(3,n-1),\ldots , ((n+1)/2,(n+3)/2)\}$. For every $i \in \{2,n\}$ let us denote the rotated versions of $M_1$ by $M_{i}:= \{(2+(i-1),n+(i-1)),(3+(i-1),n-1+(i-1)),\ldots , ((n+1)/2+(i-1),(n+3)/2+(i-1))\}$ where everything is understood modulo $n$. Let $\mathcal{M}:= \{M_1, \ldots M_n\}$. Observe that the matchings in $\mathcal{M}$ are pairwise edge disjoint. The union of the matchings $H_1=M_1 \cup M_{(n+1)/2}$ is a Hamiltonian path. For each $2 \leq i \leq n$ let $H_{i}:= M_{i} \cup M_{(n-1)/2+(i-1)}$.  We claim that the set $\mathcal{F}:= \{H_1, \ldots H_n\}$ satisfies the conditions of our claim. The $H_i$ are Hamiltonian paths since for every $2 \leq i \leq n$, $H_i$ is just the rotated version of $H_1$. If two Hamiltonian paths in $\mathcal{F}$ share an edge, then they share a matching in $\mathcal{M}$ since the matchings in $\mathcal{M}$ were edge disjoint. Thus every Hamiltonian path $H_i$ in $\mathcal{F}$ corresponds to a set of two matchings, that are $\{M_{1+(i-1)},M_{(n+1)/2+(i-1)}\}$ and two Hamiltonian paths share an edge if and only if their corresponding set of matchings intersect. Thus it is easy to see that, the intersection graph of the Hamiltonian paths in $\mathcal{F}$ is indeed a cycle of length $n$.  
\end{proof}

\begin{thm} \label{thm:first}
For every positive integer $n$, we have $P(n,2)=(n-1)!.$
\end{thm}
\begin{proof}
The lower bound $(n-1)! \leq P(n,2)$ follows from the observation that the number of Hamiltonian paths of $K_n$ containing a fixed edge is $(n-1)!$.

The upper bound when $n$ is even: By Claim \ref{claim:hamdecomp}, there is a decomposition of the edges of $K_n$ into Hamiltonian paths. Let us fix such a decomposition $\mathcal{F}$. Let us permute the ground set under $\mathcal{F}$: For every $\pi$ permutation of the set $[n]$, we define $\mathcal{F}_{\pi}$ to be the Hamiltonian path decomposition of $K_n$ obtained from $\mathcal{F}$ by relabelling the vertices of the ground set from $\{1,2,\ldots, n\}$ to $\{\pi(1),\pi(2), \ldots , \pi(n)\}.$ Let $\mathcal{H}$ be a family of Hamiltonian paths such that every pair of Hamiltonian paths in $\mathcal{H}$ have an edge in common. Since for every $\pi$,  $\mathcal{F}_{\pi}$ consists of edge disjoint Hamiltonian paths, clearly we have   $|\mathcal{H} \cap \mathcal{F}_{\pi}| \leq 1 $. Moreover, for every Hamiltonian path $H$ (not necessarily in $\mathcal{H}$), there are exactly $n$ permutations $\pi_1, \ldots , \pi_{n}$ such that for all $i \in \{1,\ldots ,n\}$,  $H \in \mathcal{F}_{\pi_i}$ (Every Hamiltonian path in $\mathcal{F}$ can be relabelled to $H$ in exactly two ways). Therefore  

\[ n |\mathcal{H}| \leq \sum_{\pi \in S_n} |\mathcal{H} \cap \mathcal{F}_{\pi}| \leq n!. 
 \] Thus the proof of the case where $n$ is even is complete.
 
The upper bound when $n$ is odd: We proceed similarly to the case when $n$ is even. Let $\mathcal{F}$ be a family of Hamiltonian paths as in Claim \ref{claim:hamcover}. for every permutation $\pi $, we define $\mathcal{F}_{\pi}$ to be the Hamiltonian path decomposition of $K_n$ obtained from $\mathcal{F}$ by relabelling the vertices of the ground set from $\{1,2,\ldots, n\}$ to $\{\pi(1),\pi(2), \ldots , \pi(n)\}.$ Let $\mathcal{H}$ be a family of Hamiltonian paths such that every pair of Hamiltonian paths in $\mathcal{H}$ have an edge in common. By Claim \ref{claim:hamcover}, for every permutation $\pi$, we have   $|\mathcal{H} \cap \mathcal{F}_{\pi}| \leq 2 $. For every Hamiltonian path $H$ (not necessarily in $\mathcal{H}$), there are exactly $2n$ permutations $\pi_1, \ldots , \pi_{n}$ such that for all $i \in \{1,\ldots ,n\}$,  $H \in \mathcal{F}_{\pi_i}$ (Every Hamiltonian path in $\mathcal{F}$ can be relabelled to $H$ in exactly two ways). Therefore 
\[ 2n |\mathcal{H}| \leq \sum_{\pi \in S_n} |\mathcal{H} \cap \mathcal{F}_{\pi}| \leq 2n! 
\]  and the proof is complete.
\end{proof}

\begin{remark}
Theorem \ref{thm:first} was also proven independently by Casey Tompkins, \cite{casey}.
\end{remark}

\section{The lower bound for $k = 2^\ell+1 $} \label{sec:lower_bound}
We say that two Hamiltonian paths are $k$-neighbor separated if their union contains a $k$-cycle, such that one of them has at least $k$ edges in that cycle. It is easy to see that two permutations are $k$-neighbor separated if and only if their corresponding Hamiltonian paths are $k$-neighbor separated. We will construct Hamiltonian paths instead of permutations. The main advantage of this is that we can draw helpful figures. We construct a suitable system of Hamiltonian paths in two steps.
First we build an appropriate set of  so-called ``labelled graphs'', then we construct numerous Hamiltonian paths from each labelled graph.

\subsection{Labelled graphs}

\begin{defi}
A labelled graph is a graph where each edge gets a label from $\{a,b\}$.
\end{defi}

We will be interested in disconnected labelled graphs where every connected component is a \textit{labelled grid} that is defined as follows.

\begin{defi}
A labelled grid of width $w$ and height $h$ is a graph on $wh$ vertices of the form $(i,j)$ where $1 \leq i \leq w$ and $1 \leq j \leq h$. Two vertices are connected by an edge of label {\em a} if they differ only in their first coordinate and the difference there equals one. Similarly, two vertices are connected by an edge of label {\em b} if they differ only in their second coordinate and the difference there is one, see Figure \ref{grids}.
\end{defi}

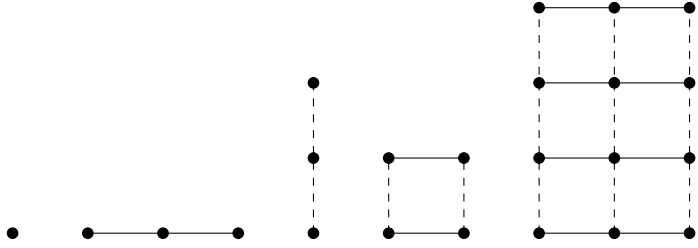
\begin{figure}[htbp]
\begin{center}
\begin{tikzpicture}

\filldraw[black] (1,1) circle (2pt);

\begin{scope}[shift={(1,1)}]
\filldraw[black] (1,0) circle (2pt);
\filldraw[black] (2,0) circle (2pt);
\filldraw[black] (3,0) circle (2pt);
\draw(1,0)--(3,0);
\end{scope}

\begin{scope}[shift={(5,0)}]
\filldraw[black] (0,1) circle (2pt);
\filldraw[black] (0,2) circle (2pt);
\filldraw[black] (0,3) circle (2pt);
\draw[dashed](0,1)--(0,3);
\end{scope}

\begin{scope}[shift={(6,0)}]
\filldraw[black] (0,1) circle (2pt);
\filldraw[black] (0,2) circle (2pt);
\filldraw[black] (1,1) circle (2pt);
\filldraw[black] (1,2) circle (2pt);
\draw (0,1)--(1,1);
\draw (0,2)--(1,2);
\draw[dashed](1,1)--(1,2);
\draw[dashed](0,1)--(0,2);
\end{scope}

\begin{scope}[shift={(8,0)}]
\filldraw[black] (0,1) circle (2pt);
\filldraw[black] (0,2) circle (2pt);
\filldraw[black] (0,3) circle (2pt);
\filldraw[black] (0,4) circle (2pt);
\filldraw[black] (1,1) circle (2pt);
\filldraw[black] (1,2) circle (2pt);
\filldraw[black] (1,3) circle (2pt);
\filldraw[black] (1,4) circle (2pt);
\filldraw[black] (2,1) circle (2pt);
\filldraw[black] (2,2) circle (2pt);
\filldraw[black] (2,3) circle (2pt);
\filldraw[black] (2,4) circle (2pt);
\draw[dashed] (0,1) -- (0,4);
\draw[dashed] (1,1) -- (1,4);
\draw[dashed] (2,1) -- (2,4);
\draw (0,1) -- (2,1);
\draw (0,2) -- (2,2);
\draw (0,3) -- (2,3);
\draw (0,4) -- (2,4);
\end{scope}

\end{tikzpicture}
\end{center}
\label{grids}
\caption{In our figures we draw the edges of label $a$ as ordinary edges and the edges of label $b$ as dashed ones. Here we see five labelled grids, their width and height in a $(w,h)$ format from the left to right is: $(1,1),$ $(3,1),$
$(1,3),$ $(2,2)$ and $(3,4).$ }
\end{figure}

\begin{defi}
A $w$-labelled graph  is the disjoint union of an isolated vertex and some labelled grids of width exactly $w$.
\end{defi}

The following construction describes, how to build Hamiltonian paths from $w$-weighed graphs. (In this paper we will build Hamiltonian paths only from $w$-labelled graphs for some integer $2\leq w$.)
\medskip

\noindent {\bf Z-swapping construction.} Let $2 \leq w$ and $W$ be a $w$-labelled graph and let $g$ be the number of labelled grids of width $w$ in $W$. We construct $2^g$ Hamiltonian paths from $W$ as follows. Fix an order of a components of $W$ where the first component is the isolated vertex. The Hamiltonian paths that we construct will be indexed by $0-1$ sequences of length $g$. Each Hamiltonian path starts at the isolated vertex of $W$, and visits the labelled grids according to the fixed order. At the $i$-th labelled grid if the $i$-th element in the $0-1$ sequence of the path is $1$ then the path starts at the top right vertex, if it is zero, it starts from the top left vertex. The path contains every edge of label $a$ of the labelled grid, and it visits the rows of the grid from the top to the bottom. Moreover if it started in the top right corner then it traverses every row from the right to the left, otherwise it traverses every row from left to the right, see Figure~\ref{connecting}.

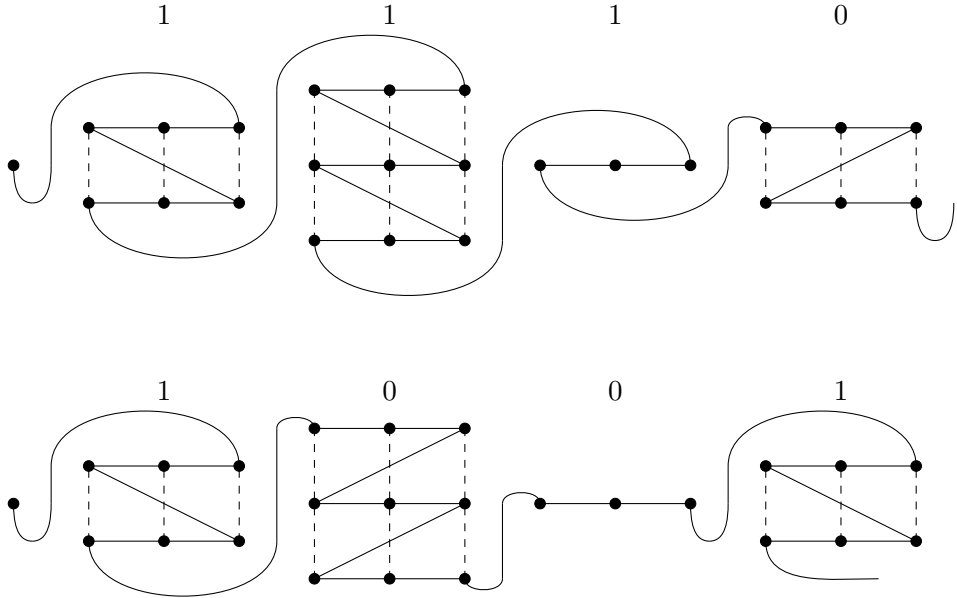
\begin{figure}
\begin{center}
\begin{tikzpicture}
\node() at (2,1.5){$1$};
\node() at (5,1.5){$0$};
\node() at (8,1.5){$0$};
\node() at (11,1.5){$1$};

\filldraw[black] (0,0) circle (2pt);

\begin{scope}[shift={(1,-0.5)}]
\filldraw[black] (0,0) circle (2pt);
\filldraw[black] (1,0) circle (2pt);
\filldraw[black] (2,0) circle (2pt);
\filldraw[black] (0,1) circle (2pt);
\filldraw[black] (1,1) circle (2pt);
\filldraw[black] (2,1) circle (2pt);
\draw (0,0)--(2,0);
\draw (0,1)--(2,1);
\draw[dashed] (0,0)--(0,1);
\draw[dashed] (1,0)--(1,1);
\draw[dashed] (2,0)--(2,1);
\end{scope}

\begin{scope}[shift={(4,-1)}]
\filldraw[black] (0,0) circle (2pt);
\filldraw[black] (1,0) circle (2pt);
\filldraw[black] (2,0) circle (2pt);
\filldraw[black] (0,1) circle (2pt);
\filldraw[black] (1,1) circle (2pt);
\filldraw[black] (2,1) circle (2pt);
\filldraw[black] (0,2) circle (2pt);
\filldraw[black] (1,2) circle (2pt);
\filldraw[black] (2,2) circle (2pt);
\draw (0,0)--(2,0);
\draw (0,1)--(2,1);
\draw (0,2)--(2,2);
\draw[dashed] (0,0)--(0,2);
\draw[dashed] (1,0)--(1,2);
\draw[dashed] (2,0)--(2,2);
\end{scope}

\begin{scope}[shift={(7,0)}]
\filldraw[black] (0,0) circle (2pt);
\filldraw[black] (1,0) circle (2pt);
\filldraw[black] (2,0) circle (2pt);
\draw (0,0)--(2,0);
\end{scope}

\begin{scope}[shift={(10,-0.5)}]
\filldraw[black] (0,0) circle (2pt);
\filldraw[black] (1,0) circle (2pt);
\filldraw[black] (2,0) circle (2pt);
\filldraw[black] (0,1) circle (2pt);
\filldraw[black] (1,1) circle (2pt);
\filldraw[black] (2,1) circle (2pt);
\draw (0,0)--(2,0);
\draw (0,1)--(2,1);
\draw[dashed] (0,0)--(0,1);
\draw[dashed] (1,0)--(1,1);
\draw[dashed] (2,0)--(2,1);
\end{scope}

\draw  (0,0) to[out=-90,in=180]  (0.25,-0.5)  ;
\draw  (0.25,-0.5) to[out=0,in=-90]  (0.5,0)  ;
\draw  (0.5,0) to[out=90,in=-90] (0.5,0.5)   ;
\draw  (0.5,0.5) to[out=90,in=90]  (3,0.5)  ;
\draw  (1,0.5)-- (3,-0.5) ;
\draw  (1,-0.5) to[out=-90,in=-90]  (3.5,-0.5)  ;
\draw (3.5,-0.5) -- (3.5,1);
\draw  (3.5,1) to[out=90,in=90]  (4,1) ;
\draw (6,1)--(4,0);
\draw (6,0)--(4,-1);
\draw  (6,-1) to[out=-90,in=-90]  (6.5,-1) ;
\draw (6.5,-1) -- (6.5,0);
\draw  (6.5,0) to[out=90,in=90]  (7,0) ;
\draw  (9,0) to[out=-90,in=180]  (9.25,-0.5) ;
\draw  (9.25,-0.5) to[out=0,in=-90]  (9.5,0) ;
\draw (9.5,0)-- (9.5,0.5);
\draw  (9.5,0.5) to[out=90,in=90]  (12,0.5) ;
\draw (10,0.5)--(12,-0.5);
\draw  (10,-0.5) to[out=-90,in=180]  (11.5,-1) ;

\begin{scope}[shift={(0,4.5)}]

\node() at (2,2){$1$};
\node() at (5,2){$1$};
\node() at (8,2){$1$};
\node() at (11,2){$0$};

\filldraw[black] (0,0) circle (2pt);

\begin{scope}[shift={(1,-0.5)}]
\filldraw[black] (0,0) circle (2pt);
\filldraw[black] (1,0) circle (2pt);
\filldraw[black] (2,0) circle (2pt);
\filldraw[black] (0,1) circle (2pt);
\filldraw[black] (1,1) circle (2pt);
\filldraw[black] (2,1) circle (2pt);
\draw (0,0)--(2,0);
\draw (0,1)--(2,1);
\draw[dashed] (0,0)--(0,1);
\draw[dashed] (1,0)--(1,1);
\draw[dashed] (2,0)--(2,1);
\end{scope}

\begin{scope}[shift={(4,-1)}]
\filldraw[black] (0,0) circle (2pt);
\filldraw[black] (1,0) circle (2pt);
\filldraw[black] (2,0) circle (2pt);
\filldraw[black] (0,1) circle (2pt);
\filldraw[black] (1,1) circle (2pt);
\filldraw[black] (2,1) circle (2pt);
\filldraw[black] (0,2) circle (2pt);
\filldraw[black] (1,2) circle (2pt);
\filldraw[black] (2,2) circle (2pt);
\draw (0,0)--(2,0);
\draw (0,1)--(2,1);
\draw (0,2)--(2,2);
\draw[dashed] (0,0)--(0,2);
\draw[dashed] (1,0)--(1,2);
\draw[dashed] (2,0)--(2,2);
\end{scope}

\begin{scope}[shift={(7,0)}]
\filldraw[black] (0,0) circle (2pt);
\filldraw[black] (1,0) circle (2pt);
\filldraw[black] (2,0) circle (2pt);
\draw (0,0)--(2,0);
\end{scope}

\begin{scope}[shift={(10,-0.5)}]
\filldraw[black] (0,0) circle (2pt);
\filldraw[black] (1,0) circle (2pt);
\filldraw[black] (2,0) circle (2pt);
\filldraw[black] (0,1) circle (2pt);
\filldraw[black] (1,1) circle (2pt);
\filldraw[black] (2,1) circle (2pt);
\draw (0,0)--(2,0);
\draw (0,1)--(2,1);
\draw[dashed] (0,0)--(0,1);
\draw[dashed] (1,0)--(1,1);
\draw[dashed] (2,0)--(2,1);
\end{scope}

\draw  (0,0) to[out=-90,in=180]  (0.25,-0.5)  ;
\draw  (0.25,-0.5) to[out=0,in=-90]  (0.5,0)  ;
\draw  (0.5,0) to[out=90,in=-90] (0.5,0.5)   ;
\draw  (0.5,0.5) to[out=90,in=90]  (3,0.5)  ;
\draw  (1,0.5)-- (3,-0.5) ;
\draw  (1,-0.5) to[out=-90,in=-90]  (3.5,-0.5)  ;
\draw (3.5,-0.5) -- (3.5,1);
\draw  (3.5,1) to[out=90,in=90]  (6,1) ;
\draw (4,1)--(6,0);
\draw (4,0)--(6,-1);
\draw  (4,-1) to[out=-90,in=-90]  (6.5,-1) ;
\draw (6.5,-1) -- (6.5,0);
\draw  (6.5,0) to[out=90,in=90]  (9,0) ;
\draw  (7,0) to[out=-90,in=-90]  (9.5,0) ;
\draw (9.5,0)-- (9.5,0.5);
\draw  (9.5,0.5) to[out=90,in=90]  (10,0.5) ;
\draw (12,0.5)--(10,-0.5);
\draw  (12,-0.5) to[out=-90,in=180]  (12.25,-1) ;
\draw  (12.25,-1) to[out=0,in=-90]  (12.5,-0.5) ;
\end{scope}

\end{tikzpicture}
\end{center}
\caption{How the $0-1$ sequence influences the 'Z' and reversed 'Z' shapes. In this example $w=3$.}
\label{connecting}
\end{figure}

\begin{claim} \label{hamiltonsize}
The $Z$-swapping construction applied to a $w$-labelled graph with $g$ labelled grids (not counting the isolated vertex) produces $2^g$ pairwise $w+1$-neighbor separated Hamiltonian paths.
\end{claim}
\begin{proof}
The number of Hamiltonian paths is immediate from the definition. For the $w+1$-neighbor separatedness, consider two different Hamiltonian paths. These correspond to two $0-1$ sequences. If the first difference of these sequences occurs at the $i$-th coordinate, then the Hamiltonian paths differ at the $i$-th labelled grid. This means that the last vertex of the $(i-1)$-th grid (bottom left or bottom right, depending on the $0-1$ sequences of the grids) is connected to the top left of the $i$-th grid in one path, and to the top right in the other. Therefore the last vertex of the $(i-1)$-th grid and the $w$ vertices at the top of the $i$-th grid form a $w+1$-cycle in the union of the two Hamiltonian paths and actually both Hamiltonian paths contain $w$ edges from this cycle. Therefore the Hamiltonian paths are $w+1$-neighbor separated.
\end{proof}

Observe that when we build a Hamiltonian path $H$ from a  labelled graph using the $Z$-swapping construction, the vertices that are connected by an edge of label $a$ in $W$ are neighbors in $H$. The vertices that are connected by an edge of label $b$ in $W$ will be separated by $w-1$ other vertices in $H$ if the edge was in a grid of width $w$. Suppose that we have two labelled graphs $W_1,W_2$ that contain only grids of width $w$ (plus an isolated vertex), and both contain an edge $e$ which gets label $a$ in $W_1$ and label $b$ in $W_2$. Observe that the Hamiltonian paths constructed by the $Z$-swapping construction from $W_1$ and $W_2$ altogether form a pairwise $w+1$-neighbor separated family. This motivates the following definition.

\begin{defi}
Two labelled graphs are compatible if they share an edge which has different labels in the two labelled graphs.
\end{defi}

\begin{remark}
Note that for two graphs to be compatible we do not assume anything about the actual structure of the graphs, only that they share an edge with different labels. In Section \ref{sec:classic} we discuss the connection between the compatibility of labelled graphs and some Bollobás-type problems.
\end{remark}

In this paper the main connection between labelled graphs and Hamiltonian paths is the $Z$-swapping construction. From now on we work with labelled graphs instead of Hamiltonian paths. Let us start to build a large family of  $2$-labelled graphs. We will start every subsequent building process with this family.  We will build labelled grids of width $2$ using labelled graphs of width $1$ as building blocks.  Step by step, we will construct larger and larger families of pairwise compatible labelled graphs. However in the intermediate steps the labelled graphs are not yet $2$-labelled graphs since they might contain grids of width $1$.

\subsection{Merging operations}

We describe a procedure that takes a special labelled graph $W$ as input, and produces two compatible labelled graphs $W_1$ and $W_2$ by merging some of the labelled grids in $W$ in two different ways.

\medskip

\noindent {\bf Simple merging operations.} Suppose that we have a labelled graph $W$ that contains two labelled grids $g_1$ and $g_2$ of the same width and height. Let us denote the corner points of $g_1$ by $\{a, b, c, d\}$ and the corner points of $g_2$ by $\{e, f, g, h\}$, see Figure~\ref{fig:merge}.From $W$ we can construct two new labelled graphs $W_1$ and $W_2$, that differ from $W$ only in the way $g_1$ and $g_2$ are merged. We describe these merging methods as follows.
\begin{enumerate}
	\item {\bf Adjoining.} In $W_1$ the two grids are merged in such a way that the resulting labelled grid is two times as wide as the original ones, so we connect the rightmost points ($b-d$ side) of $g_1$ to the corresponding leftmost points ($e-g$ side) of $g_2$ by edges of label $a$, see Figure \ref{fig:merge}.
	\item {\bf Rotating down.}	$W_2$ is obtained from $W$ by ``rotating down'' the second labelled grid: the two grids are merged in such a way that the resulting labelled grid is two times as high as the original ones, we connect the vertices on the $c-d$ side to the vertices on the $h-g$ side by edges of label $b$, see Figure \ref{fig:merge}.
\end{enumerate}
 Now $W_1$ and $W_2$ are compatible (by the edge $dg$). Moreover, both $W_1$ and $W_2$ are compatible with any other graph that was originally compatible with $W$, since they contain $W$ as a subgraph with unchanged labels.
\medskip

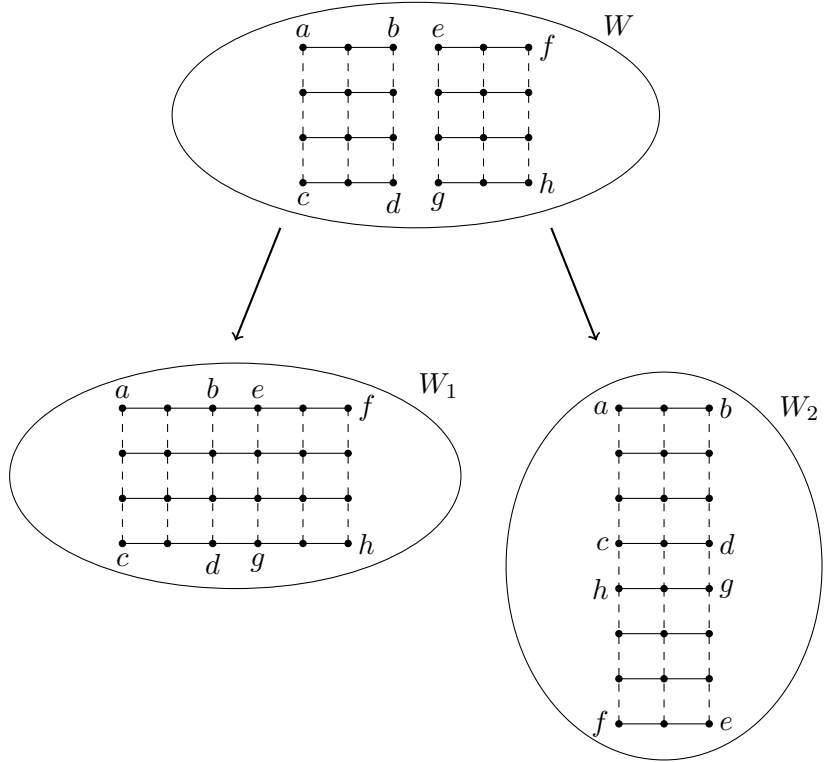
\begin{figure}[t]
\begin{center}
\begin{tikzpicture}[scale=0.6]

\begin{scope}[shift={(3,6)}] 


\draw (4.5,-0.5) ellipse (5.4cm and 2.5cm);

\node() at (9,1.5){$W$};

\filldraw[black] (2,1) circle (2pt)node[anchor=south] {$a$};
\filldraw[black] (2,0) circle (2pt);
\filldraw[black] (2,-1) circle (2pt);
\filldraw[black] (2,-2) circle (2pt)node[anchor=north] {$c$};
\filldraw[black] (3,1) circle (2pt);
\filldraw[black] (3,0) circle (2pt);
\filldraw[black] (3,-1) circle (2pt);
\filldraw[black] (3,-2) circle (2pt);
\filldraw[black] (4,1) circle (2pt)node[anchor=south] {$b$};
\filldraw[black] (4,0) circle (2pt);
\filldraw[black] (4,-1) circle (2pt);
\filldraw[black] (4,-2) circle (2pt)node[anchor=north] {$d$};
\draw (2,1) -- (4,1) ;
\draw (2,0) -- (4,0) ;
\draw (2,-1) -- (4,-1) ;
\draw (2,-2) -- (4,-2) ;
\draw[dashed](2,1) -- (2,-2);
\draw[dashed](3,1) --(3,-2);
\draw[dashed](4,1) --(4,-2);

 \begin{scope}[shift={(3
,0)}] 

 \filldraw[black] (2,1) circle (2pt)node[anchor=south] {$e$};
\filldraw[black] (2,0) circle (2pt);
\filldraw[black] (2,-1) circle (2pt);
\filldraw[black] (2,-2) circle (2pt)node[anchor=north] {$g$};
\filldraw[black] (3,1) circle (2pt);
\filldraw[black] (3,0) circle (2pt);
\filldraw[black] (3,-1) circle (2pt);
\filldraw[black] (3,-2) circle (2pt);
\filldraw[black] (4,1) circle (2pt)node[anchor=west] {$f$};
\filldraw[black] (4,0) circle (2pt);
\filldraw[black] (4,-1) circle (2pt);
\filldraw[black] (4,-2) circle (2pt)node[anchor=west] {$h$};
\draw (2,1) -- (4,1) ;
\draw (2,0) -- (4,0) ;
\draw (2,-1) -- (4,-1) ;
\draw (2,-2) -- (4,-2) ;
\draw[dashed](2,1) -- (2,-2);
\draw[dashed](3,1) --(3,-2);
\draw[dashed](4,1) --(4,-2);

   \end{scope}

\draw[->, thick] (1.5,-3) -- (0.5,-5.5);
\draw[->, thick] (7.5,-3) -- (8.5,-5.5);

\end{scope}

\begin{scope}[shift={(-1,-2)}]
\draw (4.5,-0.5) ellipse (5cm and 2.5cm);
\node() at (9,1.5){$W_1$};



\filldraw[black] (2,1) circle (2pt)node[anchor=south] {$a$};
\filldraw[black] (2,0) circle (2pt);
\filldraw[black] (2,-1) circle (2pt);
\filldraw[black] (2,-2) circle (2pt)node[anchor=north] {$c$};
\filldraw[black] (3,1) circle (2pt);
\filldraw[black] (3,0) circle (2pt);
\filldraw[black] (3,-1) circle (2pt);
\filldraw[black] (3,-2) circle (2pt);
\filldraw[black] (4,1) circle (2pt)node[anchor=south] {$b$};
\filldraw[black] (4,0) circle (2pt);
\filldraw[black] (4,-1) circle (2pt);
\filldraw[black] (4,-2) circle (2pt)node[anchor=north] {$d$};
\draw (2,1) -- (5,1) ;
\draw (2,0) -- (5,0) ;
\draw (2,-1) -- (5,-1) ;
\draw (2,-2) -- (5,-2) ;
\draw[dashed](2,1) -- (2,-2);
\draw[dashed](3,1) --(3,-2);
\draw[dashed](4,1) --(4,-2);

 \begin{scope}[shift={(3
,0)}]

 \filldraw[black] (2,1) circle (2pt)node[anchor=south] {$e$};
\filldraw[black] (2,0) circle (2pt);
\filldraw[black] (2,-1) circle (2pt);
\filldraw[black] (2,-2) circle (2pt)node[anchor=north] {$g$};
\filldraw[black] (3,1) circle (2pt);
\filldraw[black] (3,0) circle (2pt);
\filldraw[black] (3,-1) circle (2pt);
\filldraw[black] (3,-2) circle (2pt);
\filldraw[black] (4,1) circle (2pt)node[anchor=west] {$f$};
\filldraw[black] (4,0) circle (2pt);
\filldraw[black] (4,-1) circle (2pt);
\filldraw[black] (4,-2) circle (2pt)node[anchor=west] {$h$};
\draw (2,1) -- (4,1) ;
\draw (2,0) -- (4,0) ;
\draw (2,-1) -- (4,-1) ;
\draw (2,-2) -- (4,-2) ;
\draw[dashed](2,1) -- (2,-2);
\draw[dashed](3,1) --(3,-2);
\draw[dashed](4,1) --(4,-2);

   \end{scope}

   \end{scope}

\begin{scope}[shift={(10,-2)}] 

\node() at (6,1){$W_2$};

\draw (3,-2.5) ellipse (3.5cm and 4.3cm);

\filldraw[black] (2,1) circle (2pt)node[anchor=east] {$a$};
\filldraw[black] (2,0) circle (2pt);
\filldraw[black] (2,-1) circle (2pt);
\filldraw[black] (2,-2) circle (2pt)node[anchor=east] {$c$};
\filldraw[black] (3,1) circle (2pt);
\filldraw[black] (3,0) circle (2pt);
\filldraw[black] (3,-1) circle (2pt);
\filldraw[black] (3,-2) circle (2pt);

\filldraw[black] (4,1) circle (2pt)node[anchor=west] {$b$};
\filldraw[black] (4,0) circle (2pt);
\filldraw[black] (4,-1) circle (2pt);
\filldraw[black] (4,-2) circle (2pt)node[anchor=west] {$d$};
\draw (2,1) -- (4,1) ;
\draw (2,0) -- (4,0) ;
\draw (2,-1) -- (4,-1) ;
\draw (2,-2) -- (4,-2) ;

\draw[dashed](2,1) -- (2,-3);
\draw[dashed](3,1) --(3,-3);
\draw[dashed](4,1) --(4,-3);
 \begin{scope}[shift={(0
,-4)}]
 \filldraw[black] (2,1) circle (2pt)node[anchor=east] {$h$};
\filldraw[black] (2,0) circle (2pt);
\filldraw[black] (2,-1) circle (2pt);
\filldraw[black] (2,-2) circle (2pt)node[anchor=east] {$f$};
\filldraw[black] (3,1) circle (2pt);
\filldraw[black] (3,0) circle (2pt);
\filldraw[black] (3,-1) circle (2pt);
\filldraw[black] (3,-2) circle (2pt);
\filldraw[black] (4,1) circle (2pt)node[anchor=west] {$g$};
\filldraw[black] (4,0) circle (2pt);
\filldraw[black] (4,-1) circle (2pt);
\filldraw[black] (4,-2) circle (2pt)node[anchor=west] {$e$};
\draw (2,1) -- (4,1) ;
\draw (2,0) -- (4,0) ;
\draw (2,-1) -- (4,-1) ;
\draw (2,-2) -- (4,-2) ;
\draw[dashed](2,1) -- (2,-2);
\draw[dashed](3,1) --(3,-2);
\draw[dashed](4,1) --(4,-2);

   \end{scope}

    \end{scope}

\end{tikzpicture}
\end{center}
\caption{The two labelled grids with corners $\{a,b,c,d\}$ and $\{e,f,g,h\}$ of width $3$ and height $4$ can be merged in two different ways, see $W_1$ for {\bf Adjoining} and $W_2$ for {\bf Rotating down}. The two resulting labelled graphs are compatible by the edge $dg$.}
\label{fig:merge}
\end{figure}

\noindent {\bf Multiple merging operation.}
Let  $1 \leq w,h,m $ be positive integers and $W$ be an arbitrary labelled graph. Let $M$ be a set of labelled grids of $W$ that consists of $2m$ labelled grids of equal width $w$ and equal height $h$.  From the pair $(W,M)$ the \textit{multiple merging operation} produces $\binom{m}{\lfloor \frac{m}{2} \rfloor}$ pairwise compatible labelled graphs as follows.

We pair up the $2m$ grids into $m$ ordered pairs and we also fix an arbitrary ordering of these ordered pairs.  We build a family of labelled graphs that is indexed by the subsets of $[m]$ of size $\lfloor \frac{m}{2} \rfloor$. For such a subset $S \subset [m]$, the corresponding labelled graph is formed as follows. We apply a simple merging operation on each pair, for the $i$-th pair: If  $i \in S$ then we merge the $i$-th ordered pair of labelled grids by {\bf Adjoining}. If $i \notin S$, we merge by {\bf Rotating down}.

\noindent{\em Note.} For every ordered pair of labelled grids, when we apply the simple merging operations we always use the same labelling of the corners of the grids with the labels $a,b,c,d,e,f,g,h$. This will be necessary to ensure compatibility of the resulting graphs. 
\medskip

The multiple merging operation indeed produces $\binom{m}{\lfloor \frac{m}{2} \rfloor}$ pairwise compatible labelled graphs, since two such labelled graphs correspond to two different subsets $S_1, S_2$ of $[m]$. If $j \in S_1 \triangle S_1$, then the $j$-th ordered pair is merged in different ways in these labelled graphs. This ensures compatibility as can be seen in Figure \ref{fig:merge}. All the labelled graphs that are built during this operation are isomorphic to each other. 
\medskip

\noindent{\em Note.} Instead of using subsets $S \subset [m]$ that are of size $\lfloor \frac{m}{2} \rfloor$, we can actually use every possible subset in the multiple merging operation. The resulting labelled graphs will still be compatible and we get $2^{m}$ pairwise compatible labelled graphs, but they are not pairwise isomorphic any more. In this paper we only prove asymptotic results and by using every subset we would only gain subexponential factors. We choose to give up these subexponential gains for the property that every labelled graph is isomorphic to each other since this property makes the proofs considerably simpler. In \cite{original} we determined the exact number of pairwise $3$-neighbor separated permutations using similar techniques. In that paper, there is a part of the building process that corresponds to multiple merging operation and there we used every possible subset.
\medskip

The multiple merging operation is a useful tool to produce many pairwise compatible labelled graphs from a single one. But after we apply the multiple merging operation to a pair $(W,M)$, some of the new labelled grids will have their width doubled, and some of them their height doubled. Recall that after we build a suitable family of labelled graphs we wish to use the $Z$-swapping construction. For this we not only need that our labelled graphs are pairwise compatible but that they contain grids of the same width. The purpose of the next operation is exactly this, although instead of having the width of every labelled grid constant, we will only keep the width of the vast majority of the grids constant. 

\medskip
\noindent {\bf Width doubling operation}
Let $W$ be a labelled graph  and let $X$ be an induced subgraph of $W$ that is the disjoint union of labelled grids of the same width $w$ and height $h$. Let $|V(X)|=x$ and $g$ be the number of labelled grids of  $X$, thus $g= \frac{x}{wh}$. The width doubling operation applied to a pair $(W,X)$ produces a family of pairwise isomorphic, pairwise compatible labelled graphs $\mathcal{F}(W,X)$, where each element of  $\mathcal{F}(W,X)$ contains $g/3-O(\log{g})$ labelled grids of width $2w$, and $|\mathcal{F}(W,X)| = 2^{2g/3-O(\log(g))}$. The family $\mathcal{F}(W,X)$ is built as follows.

Let $a_0$ be the largest positive even integer so that $a_0 \leq g$ and let $X_0$ be a subgraph of $X$ that contains $a_0$ labelled grids of width $w$ and height $h$. We apply the {\bf multiple merging operation} on the pair $(W,X_0)$. Let $\mathcal{F}_1$ be the resulting family that contains $ \binom{a_0/2}{\lfloor a_0/4 \rfloor }$ pairwise compatible labelled graphs. The labelled graphs in $\mathcal{F}_1$ are pairwise isomorphic as they are the result of a multiple merging operation. Two types of new labelled grids are formed in the process: $\lfloor a_0/4 \rfloor$ grids of width $w$ and height $2h$ and $\lceil a_0/4 \rceil$ grids of width $2w$ and height $h$.

We proceed by defining families of labelled graphs inductively. Suppose that $\mathcal{F}_i$ is already defined and it contains pairwise compatible pairwise isomorphic labelled graphs and  $|\mathcal{F}_i|=  \prod_{j=0}^{i-1} \binom{a_j/2}{\lfloor a_j/4 \rfloor}$. Suppose that $\{a_j\}_{j=0}^{i}$ are already defined and $a_j$ is the largest positive even number so that $a_j \leq \lfloor a_{j-1}/4 \rfloor$. Also suppose that every labelled graph in $\mathcal{F}_i$ contains at least $\lfloor a_i/4 \rfloor$ labelled grids of width $w$ and height $2^i h$. For each labelled graph $W' \in \mathcal{F}_i$ we will apply a multiple merging operation as follows. Let $a_{i+1}$ be the largest even integer smaller than or equal to $\lfloor a_{i}/4 \rfloor$ and let $X_i=X_i(W')$ be the subgraph of $W'$ that contains $a_{i+1}$ grids of width $w$ and height $2^i h$. We apply the multiple merging operation to the pair $(W',X_i)$ to get the family of labelled graphs $\mathcal{F}_{W'}$,
\[\mathcal{F}_{i+1}:= \bigcup_{W' \in \mathcal{F}_i} \mathcal{F}_{W'}.\]

The family $\mathcal{F}_{i+1}$ consists of $  \prod_{j=0}^{i} \binom{a_j/2}{\lfloor a_j/4 \rfloor}$ pairwise compatible and pairwise isomorphic labelled graphs. Every labelled graph in $\mathcal{F}_{i+1}$ contains at least $\lfloor a_{i+1}/4 \rfloor $ labelled grids of width $w$ and height $2^{i+1}h$. Since $a_1 >a_2 > \ldots $ by definition and each $a_i$ is a positive even number, there must be a last element in the sequence $\{a_1, a_2, \ldots ,a_z\}$. \begin{center}We say that $ \mathcal{F}(W,X):=\mathcal{F}_z$ is the output of the width doubling operation! \end{center}
\medskip

\noindent{\em Note.} At every multiple merging operation, we choose $a_i$ labelled grids of a certain height and width where $a_i$ must be even. If the original number of those labelled grids were odd, there is still a ``leftover'' grid of that dimension in the next family and no further operations use this grid. Therefore after a width doubling operation, there might be a small number (at most $z$) of grids that have their original width. Thus the width doubling operation applied on the pair $(W,X)$ doubles the width of most of the grids on the vertices of the grids in $X$, but there is a small error.

Our main tool for building a large family of pairwise compatible labelled graphs is the width doubling operation. Before we use it, let us prove some key properties of it.

\begin{claim}\label{claim:size}
Let $W$ be a labelled graph and $X$ be a set of labelled grids of $W$ with the same width $w$ and height $h$. If there are $g$ grids in $X$ then the size of $\mathcal{F}(W,X)$ is at least $2^{2g/3-O(\log(g)^2)}$.
\end{claim}
\begin{proof}
By definition we have  $g-1 \leq a_0 \leq g$ and for any $1 \leq i \leq z$ , we have  $\frac{a_{i-1}}{4}-1 \leq a_i \leq \frac{a_{i-1}}{4}$. From this it follows that $g/4^i -3 \leq  a_i \leq g/4^i$ and  $ \log_{4}(g)-1  \leq z \leq \log_{4}(g)$.
By Stirling's approximation there is a constant $c$ such that for every integer $N$
\begin{equation} \label{binomlower}
c \frac{2^N}{\sqrt{N}} \leq \binom{N}{\lfloor N/2 \rfloor}.
\end{equation}
Now we get a lower bound for $|\mathcal{F}(W,X)|$ using (\ref{binomlower}) and the lower bound for $a_i$ as follows
\[|\mathcal{F}(W,X)|= \prod_{i=0}^z \binom{a_i/2}{\lfloor \frac{a_i}{4} \rfloor} \geq \prod_{i=0}^z c \frac{2^{a_i/2}}{\sqrt{a_i/2}} \geq  \prod_{i=0}^{\log_4(g)-1} 2^{g/(2 \cdot  4^{i}) - \log_2{c}-\log_2{g}}. \]
From this it is an easy task to conclude that
$$|\mathcal{F}(W,X)|\geq 2^{2g/3-O(\log^2{g})}$$ as claimed.
\end{proof}

\begin{claim}\label{claim:gridnum}
Let $W$ be a labelled graph and $X$ a set of $g$ labelled grids of $W$ with the same width $w$ and height $h$. Every labelled graph in $\mathcal{F}(W,X)$ contains at least $g/3-O(\log{g})$ labelled grids of width $2w$ on the vertices of the grids in $X$.
\end{claim}
\begin{proof}
At the $i$-th merging operation $(0 \leq i \leq z-1)$ exactly $\lceil \frac{a_i}{4} \rceil$ labelled grids of width $2w$ were created, and we did not change these grids after their creation. Thus the number of grids of width $2w$ is
$$\sum_{i=0}^{z-1} \left\lceil \frac{a_i}{4} 	\right\rceil \geq \sum_{i=0}^{\log_4{g}-2} \frac{g}{4^{i+1}}-4 = g/3-O(\log{g}). $$
\end{proof}

\begin{claim}\label{claim:leftover} Let $W$ be a labelled graph and $X$ a set of $g$ labelled grids of $W$ with the same width $w$ and height $h$. For every labelled graph $W' \in \mathcal{F}(W,X)$, the number of vertices in $W'$ that are contained in labelled grids of width $w$ (``leftover grids'') and that are contained in the grids of $X$ in the original labelled graph $W$ is at most $O(wh \sqrt{g})$.
\end{claim}
\begin{proof}
In a labelled graph of $\mathcal{F}(W,X)$ those labelled grids of width $w$ that are on the vertices of the grids in $X$ are all formed in the following way. We had to chose an even number of grids of the same width and height for a  multiple merging operation, but there was an odd number of these grids.  In these cases the grid that was not used in the multiple merging operation had height $2^i h$ and width $w$ for some $(1 \leq i \leq z)$. Hence the number of vertices in such a grid is $hw 2^i$. At every multiple merging operation there was at most one such grid, thus the number of vertices in all of these grids is at most
$$ \sum_{i=1}^{\log_4 (g)}wh 2^i =wh O(\sqrt{g}). $$
\end{proof}

\subsection{The lower bound for $K=1$}

Now we are ready to prove Theorem~\ref{thm:main}. First we present the proof for $\ell=1$, then we proceed by induction. The rough structure of the proof can be seen in Figure \ref{fig:main_idea}. We note that the $\ell=1$ case follows from the main result of \cite{original}, however on one hand we need a different construction for the induction and on the other hand, here the construction is much simpler since we only prove an asymptotic result.

\bigskip

\noindent \textit{Proof of Theorem~\ref{thm:main} for k=2.} Let $\mathcal{F}_0$ be the family that contains only the empty (labelled) graph $W_0$ on $n$ vertices. We think of $W_0$ as a labelled graph that contains an isolated vertex $v$ and $n-1$ labelled grids of width $1$ and height $1$. Our aim is to produce many labelled graphs with labelled grids of width $2$ so we wish to double their width. Let $X$ be a subset of the vertices of $W_0$ of size $n-1$. Now let us use the width doubling operation on the pair $(W_0,X)$ to get the family of labelled graphs $\mathcal{F}=\mathcal{F}(W_0,X)$. By Claim \ref{claim:leftover} every labelled graph in $\mathcal{F}$ contains at most $O(\sqrt{n})$ vertices that are in labelled grids of width $1$. Therefore adding at most  $O(\sqrt{n})$ new vertices every labelled graph in $\mathcal{F}$ can be completed to a $2$-labelled graph by adding only additional edges to widen the grids. Let us denote the resulting family of $2$-labelled graphs by $\mathcal{F}^1$. By Claim \ref{claim:size}

\[ |\mathcal{F}^1| \geq 2^{2n/3-O(\log(n)^2)}. \]

We apply the Z-swapping construction to $\mathcal{F}^1$. From each labelled graph in $\mathcal{F}^1$ we get $2^{n/3-O(\log{n})}$ Hamiltonian paths by Claim \ref{hamiltonsize} and Claim \ref{claim:gridnum}. Therefore altogether we get $2^{n-O((\log{n})}$ Hamiltonian paths on $n+O(\sqrt{n})$ vertices and the proof is complete.
\hfill\(\qedsymbol\)

\subsection{The lower bound for $k=2^\ell $}

\begin{figure}
\begin{center}
\begin{tikzpicture}[scale=0.925]

\begin{scope}[shift={(0,5)}]
	 \begin{scope}[shift={(6.3,0)}]
\node() at (2,2.3){$\mathcal{F}^0$};
\draw (0,0) rectangle (4,2);
    \node() at (2,1.5){Contains only $W_0$};
    \node() at (2,1){the edgeless graph};
    \node() at (2,0.5){on $n$ vertices.};
\end{scope}
	\draw[ultra thick,->] (7,-0.5) -- (7,-2.6);
	\node() at (11.5,-1.5){One {\bf width doubling operation.}};
	
\end{scope}

\begin{scope}[shift={(0,0)}]
  \begin{scope}[shift={(6.3,0)}]
    \node() at (2,2.3){$\mathcal{F}^* = \mathcal{F}^1$};
\draw (0,0) rectangle (4,2);
    \node() at (2,1.5){Many, pairwise};
    \node() at (2,1){compatible};
    \node() at (2,0.5){2-labelled graphs.};

\draw[ultra thick,->] (4.1,1) -- (8.5,1);
    \node() at (6.4,1.9){The $Z$-swapping};
    \node() at (6.4,1.4){construction.};

    \end{scope}

      \begin{scope}[shift={(15,0)}]
    \node() at (2,2.3){$\mathcal{H}^1$};
\draw (0,0) rectangle (4,2);
    \node() at (2,1.5){$2^{n-o(n)}$ 3-neighbor };
    \node() at (2,1){separated};
	\node() at (2,0.5){Hamiltonian paths};
    \end{scope}
	\draw[ultra thick,->] (7,-0.5) -- (7,-2.6);
	\node() at (11.5,-1.5){$O(\log{n})$ {\bf width doubling operations.}};
\end{scope}

\begin{scope}[shift={(0,-5)}]
  \begin{scope}[shift={(6.3,0)}]
    \node() at (2,2.3){$\mathcal{F}^2$};
\draw (0,0) rectangle (4,2);
    \node() at (2,1.5){Many, pairwise};
    \node() at (2,1){compatible};
    \node() at (2,0.5){4-labelled graphs.};

\draw[ultra thick,->] (4.1,1) -- (8.5,1);
    \node() at (6.4,1.9){The $Z$-swapping};
    \node() at (6.4,1.4){construction.};

    \end{scope}

      \begin{scope}[shift={(15,0)}]
    \node() at (2,2.3){$\mathcal{H}^2$};
\draw (0,0) rectangle (4,2);
    \node() at (2,1.5){$2^{n-o(n)}$ 5-neighbor };
    \node() at (2,1){separated};
	\node() at (2,0.5){Hamiltonian paths};
    \end{scope}
	\draw[ultra thick,->] (7,-0.5) -- (7,-2.6);
	\node() at (11.5,-1.5){$O(\log^2{n})$ {\bf width doubling operations.}};
\end{scope}

\node() at (7,-9){$\ldots$};

	\draw[ultra thick,->] (7,-10.5) -- (7,-12.6);
	\node() at (11.5,-11.5){$O(\log^K{n})$ {\bf width doubling operations.}};

\begin{scope}[shift={(0,-15)}]
  \begin{scope}[shift={(6.3,0)}]
    \node() at (2,2.3){$\mathcal{F}^K$};
\draw (0,0) rectangle (4,2);
    \node() at (2,1.5){Many, pairwise};
    \node() at (2,1){compatible};
    \node() at (2,0.5){$2^\ell$-labelled graphs.};

\draw[ultra thick,->] (4.1,1) -- (8.5,1);
    \node() at (6.4,1.9){The $Z$-swapping};
    \node() at (6.4,1.4){construction.};

    \end{scope}

      \begin{scope}[shift={(15,0)}]
    \node() at (2,2.3){$\mathcal{H}^K$};
\draw (-0.1,-0.1) rectangle (4.1,2.1);
    \node() at (2,1.5){$2^{n-o(n)}$ $2^\ell$+1-neighbor };
    \node() at (2,1){separated};
	\node() at (2,0.5){Hamiltonian paths};
    \end{scope}
\end{scope}

\end{tikzpicture}
\end{center}
\caption{The induction steps.}
\label{fig:main_idea}
\end{figure}
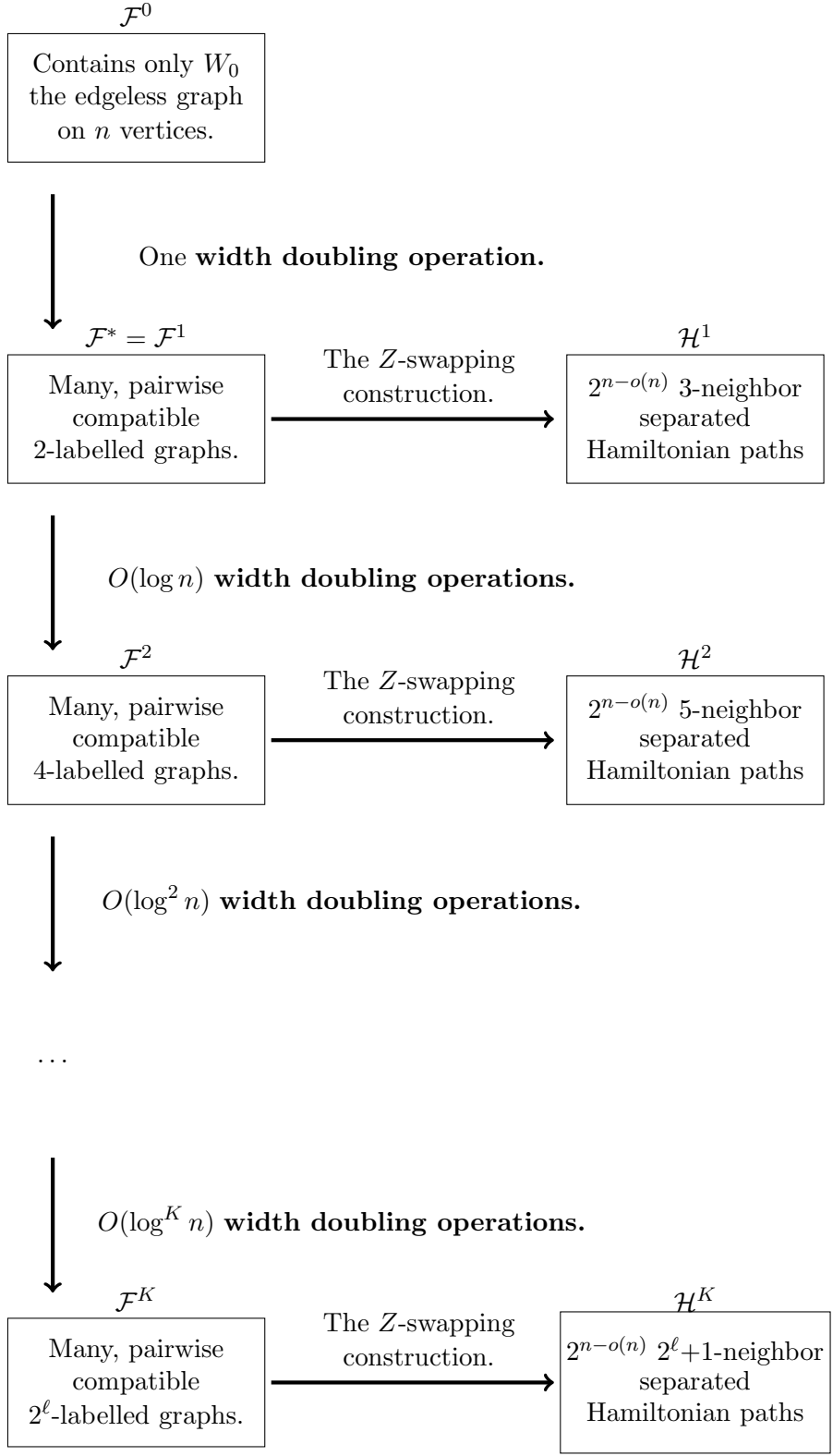

To prove Theorem~\ref{thm:main} for $\ell = 1$ we used one  width doubling operation, for general $\ell$ we will use many, see Figure \ref{fig:main_idea}. To simplify the proof we introduce the complete width doubling operation.
\medskip

\noindent {\bf Complete width doubling operation} Let $W$ be a labelled graph . The complete width doubling operation applied to $W$ produces a family of pairwise isomorphic and pairwise compatible labelled graphs $\mathcal{C}(W)$ as follows. We partition the grids of $W$ according to their dimensions (width and height). Let the classes of the partition be $P_{1},P_{2}, \ldots P_{m}$. Now let $\mathcal{C}^0:= \{W\}$ and for each $1 \leq i \leq m $ let

\[ \mathcal{C}^i:= \bigcup_{W_j \in \mathcal{C}^{i-1}} \mathcal{F}(W_j,P_i) \]
then $\mathcal{C}(W):=\mathcal{C}^{m}.$

We will also use the following claim which intuitively says that a width doubling operation does not produce too many new different shapes of grids.

\begin{claim}\label{claim:shapes}
Let $W$ be a labelled graph and $X$ a set of labelled grids in $W$ of the same width and height and $|X|=g$. In a labelled graph in $\mathcal{F}(W,X)$ the number of different shapes of grids on the vertices of $X$ is at most $O(\log{g})$.
\end{claim}
\begin{proof}
After every multiple merging operation the number of shapes in a labelled graph increases by at most two and we used $O(\log{g})$ multiple merging operations.
\end{proof}

\begin{claim} \label{claim:height}
Suppose that we are applying the width doubling operation to a pair $(W,X)$ where $X$ consists of labelled grids of height $\alpha$. Then in the resulting family  $\mathcal{F}(W,X)$ of labelled graphs the height of the grids on the vertices of $X$ is at most $\alpha 2^{\log_{4}{n/\alpha}}$.
\end{claim}
\begin{proof}
The number of grids of height $\alpha$ is obviously at most $n/\alpha$. Therefore during the width doubling operation we applied at most $\log_{4}(n/\alpha)$ multiple merging operations. After a multiple merging operation the height of the resulting grids is at most twice the height of the original ones.
\end{proof}

\noindent \textit{Proof of Theorem~\ref{thm:main} for every  $\ell$.}

Let $W_0$ be the empty graph on $n$ vertices and let us apply the complete width doubling operation $\ell$-times to $W_0$. Let the resulting series of families of labelled graphs be $\mathcal{F}^1,\mathcal{F}^2, \ldots ,\mathcal{F}^\ell$. We wish to transform the labelled graphs in $\mathcal{F}^\ell$ into $2^\ell$-labelled graphs. We show that this can be done by adding only $o(n)$ additional vertices and widening every labelled grid using these.

\begin{claim} \label{claim:littleon}
The number of vertices in a labelled graph in $\mathcal{F}^\ell$ which do not belong to a labelled grid of width $2^\ell$ is $o(n)$.
\end{claim}
\begin{proof}
By Claim \ref{claim:shapes} the number of different shapes of grids in $\mathcal{F}^1$ is at most $O(\log{n})$. Since a complete width doubling operation is a width doubling operation on every different shape, in  $\mathcal{F}^i$ the number of different shapes of grids is at most $O(\log{n}^i)$. Every grid in $\mathcal{F}^\ell$ that has width less than $2^\ell$, must be a ``leftover grid'' in some of the many width doubling operations during the complete width doubling operations that resulted in $\mathcal{F}^\ell$ therefore there is a single grid with this shape. Thus it is enough to bound the number of vertices in a single grid in $\mathcal{F}^\ell$.

The width of all the grids in $\mathcal{F}^\ell$ is bounded by a constant $2^\ell$, the height of the grids can be bounded by Claim \ref{claim:height} as follows.

In $\mathcal{F}^0$ the maximal height of a grid is $1$. If the maximal height of a grid in $\mathcal{F}^j$ is $\alpha=O(n^{1-\varepsilon})$ for a constant $\varepsilon=\varepsilon(j)$  then the maximal height of a grid in $\mathcal{F}^{j+1}$ is at most $\alpha 2^{\log_{4}(n/\alpha)}=\alpha \sqrt{n/\alpha}= \alpha^{1/2} n^{1/2}=O(n^{1-\varepsilon/2})$. Therefore setting $\varepsilon(0)=1$ and  $\varepsilon(j+1)=\varepsilon(j)/2$, by induction we have that in  $\mathcal{F}^\ell$ the height of the grids is at most $n^{1-\varepsilon}$ for a fixed epsilon. Therefore the total number of vertices in grids with width smaller than $2^\ell$ is at most $O(2^\ell n^{1-\varepsilon}  \log{n}^\ell) = o(n)$ as claimed.
\end{proof}

By Claim \ref{claim:littleon} with the addition of at most $o(n)$ new vertices we can build a family of $2^K$-labelled graphs from $\mathcal{F}^\ell$, let us call this new family $\mathcal{F}_{final}$.

\begin{claim}
Applying the $Z$-swapping construction to $\mathcal{F}_{final}$ we get $2^{n-o(n)}$ Hamiltonian paths.
\end{claim}
\begin{proof}
Let $\mathcal{F}$ be a family of labelled graphs which consists of pairwise isomorphic labelled graphs, let $g(\mathcal{F})$ be the number of labelled grids in $\mathcal{F}$. Let the value of such a family of labelled graph be the quantity $v(\mathcal{F}):=|\mathcal{F}|2^{g(\mathcal{F})}$. The value of the family that contains only the empty graph on $n$ vertices is clearly $2^n$. Suppose that we applied a width doubling operation to every graph in a family $\mathcal{F}$ with a set $X$ of size $g$, then by Claim \ref{claim:size} $|\mathcal{F}|$ is increased by at least $2^{2g/3-O(\log{g}^2)}$. But by  Claim \ref{claim:gridnum} the quantity $2^{g(\mathcal{F})}$ is decreased by at most $2^{2g/3+O(\log{g})}$. Therefore the value of $\mathcal{F}$ is decreased by at most $O(\log{g}^2)= O(\log{n}^2)$. Since $\mathcal{F}^\ell$ is the result of at most $O(\log{n}^\ell)$ width doubling operations, the value of $\mathcal{F}^\ell$ is at least $2^{n-O(\log{n}^{\ell+2})}=2^{n-o(n)}$. Since the number of Hamiltonian paths that we get after applying the $Z$-swapping construction to $\mathcal{F}_{final}$ is exactly $v(\mathcal{F}_{final})$ the proof is complete.

\end{proof}

\noindent Since the number of vertices of the labelled graphs in $\mathcal{F}_{final}$ is $n+o(n)$, the proof of Theorem~\ref{thm:main} is complete.

\section{Connection with classical results} \label{sec:classic}

In this short section we discuss how determining $P(n,k)$ and the compatibility of labelled graphs relates to some classical theorems.

\begin{thm}[B. Bollobás  \label{bol}] \cite{bollobas}
Let $\mathcal{H}= \{(A_1,B_1), (A_2,B_2), \ldots , (A_m,B_m)\}$ be a system of pairs where for each $i$ we have $|A_i|=a$, $|B_i| =b$. If $A_i \cap B_j = \emptyset$ if and only if $i=j$ then  
\[|\mathcal{H}| \leq \binom{a+b}{a}.\]
\end{thm}

The system of pairs $\mathcal{H}$ satisfying the conditions of Theorem \ref{bol} is called cross-intersecting. Note that the size of the ground set is not specified. Theorem \ref{bol} is sharp since we can take all possible partitions of $[a+b]$into two sets of size $a$ and $b$. Theorem \ref{bol} has numerous variants, for our purposes the following one is useful.

\begin{thm}[Zs. Tuza \label{tuz}]  \cite{tuza}
Let $\mathcal{H}= \{(A_1,B_1), (A_2,B_2), \ldots , (A_m,B_m)\}$ be a system of disjoint pairs where for every $i \neq j$ either $A_i \cap B_j \neq \emptyset$ or $A_j \cap B_i \neq \emptyset$. If for every $i$, $|A_i|=a$ and $|B_i| =b$ then
\[|\mathcal{H}| \leq \frac{(a+b)^{a+b}}{a^a b^b}.\]
\end{thm}

A system of pairs $\mathcal{H}$ satisfying the conditions of Theorem \ref{tuz} is called weakly cross-intersecting. Theorem \ref{tuz} is not known to be sharp, the current best lower bound is asymptotically $2 \binom{a+b}{a}$, although the fractional relaxation of Theorem \ref{tuz} is sharp, see \cite{kiraly}.

Suppose that we have a family $\mathcal{F}$ of labelled graphs. Let $\mathcal{G}$ be the family where we replace every labelled graph with the pair $(A,B)$ where $A$ is the set of edges that get label $a$ and $B$ is the set of edges that get label $b$. The family $\mathcal{F}$ contains pairwise compatible labelled graphs if and only if the family $\mathcal{G}$ consists of weakly cross-intersecting pairs.

The original problem of determining $P(n,k)$ can also be reformulated in terms of weakly cross-intersecting set systems. To each permutation $\pi$ we associate a pair $(A_{\pi},B_{\pi})$ as follows. Let $A_{\pi}$ contain the pairs of elements that are neighbors in $\pi$ and $B_{\pi}$ contain the pairs of elements that are separated by $k-2$  other elements of $\pi$. Two permutations are $k$-neighbor separated if and only if their associated pairs of sets are weakly cross intersecting. Observe that $|A_{\pi}|=n-1$ and $|B_{\pi}|=n-k+1$ therefore using $a=b=n$ in Theorem \ref{tuz} we obtain that $P(n,k) \leq 4^k$ for all $k>1$.

\section{Bounds for general fixed $k$} \label{sec:otherk}

In this section we improve the upper bound $P(n,k) \leq 4^n$ provided by Theorem \ref{tuz} for all $k$ (see Section \ref{sec:classic}). We conjecture that for every fixed $k$ the order of magnitude of $P(n,k)$ should be $2^{n-o(n)}$. For even $k$ we have seen that $$P(n,k) \leq 2^{n-O(log(n))}.$$ For odd $k$ we will need the entropy function $H(x_1,x_2,x_3):= \sum_{i=1}^{3}-x_i \log_2{x_i}.$ We will use the fact that the entropy function is related to the asymptotic exponent of a multinomial coefficient. By a straightforward application of the the Stirling formula one can obtain the following well known result.
\begin{claim}For every positive $x_1,x_2,x_3$ with the properties $x_1+x_2+x_3=1$ there are positive polynomials $q_1,q_2$ so that for every  $0 \leq n \in \mathbb{Z}$, \begin{equation} \frac{1}{q_1(n)} 2^{H(x_1,x_2,x_3)n} \leq \binom{n}{x_1n,x_2n,x_3n}  \leq q_2(n) 2^{H(x_1,x_2,x_3)n}. \label{multinomial}
\end{equation}
\end{claim}

\begin{thm} \label{thm:entropyupper} For $2 \leq k$ fixed, we have
\[P(n,k) \leq  2^{H\left( \frac{k-1}{2k},\frac{k-1}{2k},\frac{2}{2k}\right)n}. \]
\end{thm}
\begin{proof}
For every permutation $\pi$ let us associate a coloring of the ground set with  {\em red}, {\em green} and {\em blue} as follows.  For every $1 \leq i \leq n$ let $m=m(i)$ be such that $m \equiv i \mod{2k-2}$, and $m \in [1,2k-2]$.
\[ \text{The color of }\pi(i) \text{ is}
\left\{ \begin{array}{rl}
\text{red if } & m \in  \{2,3,\ldots, k-1\} \\
\text{green if } & m \in   \{k+1,k+3,\ldots, 2k-2\}\\
\text{blue if } & m \in   \{1, k\}. \\

\end{array} \right. \]

Observe that two permutations which correspond to the same coloring cannot be $k$-neighbor separated by the following reasoning. The colors of two neighboring elements in these permutations can be: red-red, blue-red, green-green, green-blue. But the colors of two elements that are separated by $k-2$ others can only be: red-green, blue-blue. Therefore $P(n,k)$ is at most the number of $3$-colorings of the ground set that we associated to the permutations. If $n$ is divisible by $2k$ then this number is exactly the multinomial coefficient
\[ \binom{n}{\frac{(k-2)n}{2k-2}, \frac{(k-2)n}{2k-2},\frac{2n}{2k-2}} \leq q(n) 2^{H\left( \frac{k-2}{2k-2},\frac{k-2}{2k-2},\frac{2}{2k-2}\right)n} \]

where $q(n)$ is a fixed polynomial. If $n$ is not divisible by $2k$ then this number is at most another polynomial factor away from the multinomial coefficient. All these polynomial factors can be safely ignored since a counterexample to the upper bound $2^{H\left( \frac{k-1}{2k},\frac{k-1}{2k},\frac{2}{2k}\right)n}$ could be blown up to a counterexample with a larger exponent, contradicting the fact that we have an upper bound with only an additional polynomial factor.

\end{proof}

\begin{corr} \label{corr:upper}
\[\limsup_{k \rightarrow \infty}P(k) \leq 2 \]
\end{corr}
\begin{proof} By Theorem \ref{thm:entropyupper} it is enough to check that
 \[\lim_{k \rightarrow \infty} H\left( \frac{k-2}{2k-2},\frac{k-2}{2k-2},\frac{2}{2k-2}\right)=1. \]
\end{proof}

\subsection{Asymptotic lower bound} \label{sec:asym}
\begin{thm}\label{thm:asymlower}
$$\liminf_{k \rightarrow \infty}P(k) \geq 2 $$
\end{thm}
\begin{proof}Recall that for a fixed $k_0$ and a fixed $n_0$ if there are $c^{n_0}$ pairwise $k_0$-neighbor separated Hamiltonian paths on $n_0$ vertices then by a product construction we have the lower bound $c \leq P(k_0)$. We will construct such families for $c$ arbitrarily close to $2$ for every large enough $k_0$. Let $r$ be a fixed integer and $k$ large so that $r$ divides $k$ and $G$ a graph on $k^2$ vertices that is the disjoint union of paths of $r$ vertices. Let $\mathcal{F}$ be the family of labelled graphs that can be formed from $G$ by adding labels from $\{a,b\}$ to every edge of $G$. Clearly
$$|\mathcal{F}|=2^{\frac{r-1}{r}k^2}.$$
For every  $ W \in \mathcal{F}$ we will construct a single Hamiltonian path $H_W$ so that two vertices that are connected by an edge with label $a$ in $W$ are neighbors in $H_W$ and two vertices that are connected by an edge with label $b$ in $W$ are exactly $k$ apart in $H_W$. We do this in a way that the graphs $H_W$ use at most $k^2+o(k^2)$ vertices. For each $W \in \mathcal{F}$ we wish to arrange the components of $W$ and $6kr+9r^2$ additional new vertices into a single $(k+3r) \times (k+3r)$ labelled grid. If we manage to do this, a suitable Hamiltonian path can be obtained adding an isolated vertex and using the Z-swapping construction (we actually obtain two paths but this does not change the order of magnitude). Instead of arranging the labelled paths in $W$ into a single labelled grid, we will fix a ``grid-shape'' of size $(k+3r) \times (k+3r)$ and we aim to fill this completely using the paths of $W$ and the additional isolated vertices (like a puzzle).  

Every labelled graph $W \in \mathcal{F}$ consists of vertex disjoint paths of length $r-1$ with some labelling of the edges from the set $\{a,b\}$. Observe that the number of possible labelling is at most $2^{r-1}$ (actually less than this if $r$ is at least $3$ since the reversed version of a labelling can be considered the same) which is a constant. We say that paths of the same labelling are of the same \textit{type}. As in other proofs, in a labelled grid we represent the edges that are labelled with $a$ with  horizontal edges and the ones that are labelled with $b$ with dashed vertical edges. Observe that using edges of the same type one can completely cover a diagonal strip in the large grid of width $r$, see part $a)$ of  Figure \ref{strip}.

\begin{figure}[htbp]
\begin{center}
\begin{tikzpicture}[scale=0.7]

\node[]() at (5,0){a)};
\begin{scope}[shift={(0,0)}]
\filldraw[black] (0,0) circle (2pt);
\filldraw[black] (1,0) circle (2pt);
\filldraw[black] (2,0) circle (2pt);
\filldraw[black] (2,-1) circle (2pt);
\filldraw[black] (3,-1) circle (2pt);
\filldraw[black] (3,-2) circle (2pt);
\filldraw[black] (3,-3) circle (2pt);
\filldraw[black] (3,-4) circle (2pt);
\filldraw[black] (4,-4) circle (2pt);
\filldraw[black] (5,-4) circle (2pt);
\draw (0,0)--(2,0);
\draw[dashed] (2,0)--(2,-1);
\draw (2,-1)--(3,-1);
\draw[dashed](3,-1)--(3,-4);
\draw (3,-4)--(5,-4);
\end{scope}

\begin{scope}[shift={(-1,-1)}]
\filldraw[black] (0,0) circle (2pt);
\filldraw[black] (1,0) circle (2pt);
\filldraw[black] (2,0) circle (2pt);
\filldraw[black] (2,-1) circle (2pt);
\filldraw[black] (3,-1) circle (2pt);
\filldraw[black] (3,-2) circle (2pt);
\filldraw[black] (3,-3) circle (2pt);
\filldraw[black] (3,-4) circle (2pt);
\filldraw[black] (4,-4) circle (2pt);
\filldraw[black] (5,-4) circle (2pt);
\draw (0,0)--(2,0);
\draw[dashed] (2,0)--(2,-1);
\draw (2,-1)--(3,-1);
\draw[dashed](3,-1)--(3,-4);
\draw (3,-4)--(5,-4);
\end{scope}

\begin{scope}[shift={(-2,-2)}]
\filldraw[black] (0,0) circle (2pt);
\filldraw[black] (1,0) circle (2pt);
\filldraw[black] (2,0) circle (2pt);
\filldraw[black] (2,-1) circle (2pt);
\filldraw[black] (3,-1) circle (2pt);
\filldraw[black] (3,-2) circle (2pt);
\filldraw[black] (3,-3) circle (2pt);
\filldraw[black] (3,-4) circle (2pt);
\filldraw[black] (4,-4) circle (2pt);
\filldraw[black] (5,-4) circle (2pt);
\draw (0,0)--(2,0);
\draw[dashed] (2,0)--(2,-1);
\draw (2,-1)--(3,-1);
\draw[dashed](3,-1)--(3,-4);
\draw (3,-4)--(5,-4);
\end{scope}

\begin{scope}[shift={(-3,-3)}]
\filldraw[black] (0,0) circle (2pt);
\filldraw[black] (1,0) circle (2pt);
\filldraw[black] (2,0) circle (2pt);
\filldraw[black] (2,-1) circle (2pt);
\filldraw[black] (3,-1) circle (2pt);
\filldraw[black] (3,-2) circle (2pt);
\filldraw[black] (3,-3) circle (2pt);
\filldraw[black] (3,-4) circle (2pt);
\filldraw[black] (4,-4) circle (2pt);
\filldraw[black] (5,-4) circle (2pt);
\draw (0,0)--(2,0);
\draw[dashed] (2,0)--(2,-1);
\draw (2,-1)--(3,-1);
\draw[dashed](3,-1)--(3,-4);
\draw (3,-4)--(5,-4);
\end{scope}

\begin{scope}[shift={(-4,-4)}]
\filldraw[black] (0,0) circle (2pt);
\filldraw[black] (1,0) circle (2pt);
\filldraw[black] (2,0) circle (2pt);
\filldraw[black] (2,-1) circle (2pt);
\filldraw[black] (3,-1) circle (2pt);
\filldraw[black] (3,-2) circle (2pt);
\filldraw[black] (3,-3) circle (2pt);
\filldraw[black] (3,-4) circle (2pt);
\filldraw[black] (4,-4) circle (2pt);
\filldraw[black] (5,-4) circle (2pt);
\draw (0,0)--(2,0);
\draw[dashed] (2,0)--(2,-1);
\draw (2,-1)--(3,-1);
\draw[dashed](3,-1)--(3,-4);
\draw (3,-4)--(5,-4);
\end{scope}

\begin{scope}[scale=1.2,shift={(7,-10)}]
\node[]() at (6,10){b)};
\draw (0,3) rectangle (6,9);
\draw (0,8)--(1,9);
\draw (0,7)--(2,9);
\draw (0,6)--(3,9);
\draw (0,5)--(4,9);
\draw (0,4)--(5,9);
\draw (0,3)--(6,9);

\begin{scope}[shift={(0,0)}]
\draw (0,9)--(0.25,9);
\draw (0.25,9)--(0.25,8.6);
\draw (0.25,8.6) -- (0.5,8.6);
\draw (0.5,8.6) -- (0.5,8.5);
\end{scope}

\foreach \x in {0,0.1666,...,1}
{\begin{scope}[shift={(1-\x,-\x)}]
\draw (0,9)--(0.25,9);
\draw (0.25,9)--(0.25,8.6);
\draw (0.25,8.6) -- (0.5,8.6);
\draw (0.5,8.6) -- (0.5,8.5);
\end{scope}}

\foreach \x in {0,0.1666,...,2}
{\begin{scope}[shift={(2-\x,-\x)}]
\draw (0,9)--(0.25,9);
\draw (0.25,9)--(0.25,8.6);
\draw (0.25,8.6) -- (0.5,8.6);
\draw (0.5,8.6) -- (0.5,8.5);
\end{scope}}

\foreach \x in {0,0.1666,...,1}
{\begin{scope}[shift={(3-\x,-\x)}]
\draw (0,9)--(0.25,9);
\draw (0.25,9)--(0.25,8.6);
\draw (0.25,8.6) -- (0.5,8.6);
\draw (0.5,8.6) -- (0.5,8.5);
\end{scope}}

\foreach \x in {1.1666,1.3333,...,3.1666}
{\begin{scope}[shift={(3-\x,-\x)}]
\draw (0,9)--(0,8.6);
\draw (0,8.6)--(0.3,8.6);
\draw (0.3,8.6) -- (0.3,8.5);
\draw (0.3,8.5) -- (0.5,8.5);
\end{scope}}

\end{scope}

\end{tikzpicture}
\end{center}

\caption{a) Shifted versions of paths of the same type fit together. b) We cover large pieces of the $(k+3r) \times (k+3r)$ grid using these strips. When we run out of a given type of paths, we start using another type in the same strip. We will use a few isolated vertices to fill in the gaps at the type change.  }
\label{strip}
\end{figure}
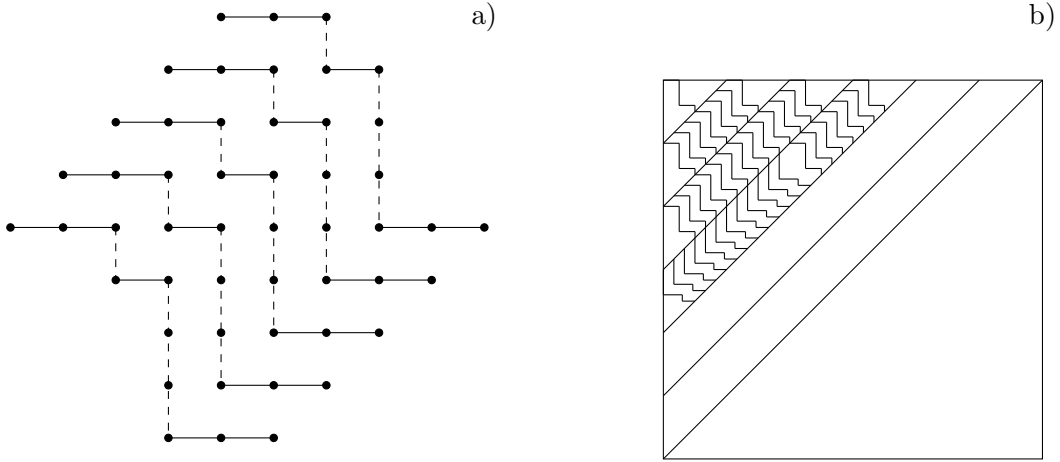

We divide the $(k+3r) \times (k+3r)$ grid into stripes of width $r$ and from top to the bottom, from left to the right we fill those stripes with labelled paths of the same type. If we run out of a type, we start using paths of another type in the same stripe as can be seen in part $b)$ of Figure \ref{strip}.
\begin{claim}For large enough $k$, every labelled path of $W$ fits into the $(k+3r) \times (k+3r)$ grid.
\end{claim}
\begin{proof} Suppose to the contrary that there are still paths from $W$, let us count the number of uncovered vertices in the grid. there are two reasons why a vertex is uncovered: Either we got to the edge of the grid and the path that would cover our vertex sticks out of the grid, or we ran out of a type and started to use another. The number of vertices that are uncovered because the edge of the grid  is too close is at most $4k r$, since the distance of these vertices from the edge of the grid is at most $r$.
The number of vertices that are uncovered because we ran out of a type and started to use another is even smaller. The number of type changes is at most $2^{r-1}$. At a single type change it is easy to see that the number of uncovered vertices is $O(r^2)$. Therefore the number of vertices that are uncovered because of a type change is at most $O(2^r r^2 )$. Thus the number of covered vertices is at least $(k+3r)^2-4kr-O(2^{r-1}r^2)=k^2+2kr-O(2^{r-1}r^2)$. Since $r$ is a constant, for large enough $k$ this quantity is larger than $k^2$, the total number of vertices in $W$ which is a contradiction.
\end{proof}

Since every labelled path in $W$ fits into the $(k+3r) \times (k+3r)$ grid, for every labelled graph $W\in \mathcal{F}$ (recall that labelled graphs in $\mathcal{F}$ have $k^2$ vertices), we can construct a Hamiltonian path on $(k+3r)^2$ vertices so that the new system of Hamiltonian paths is $(k+3r+1)$-neighbor separated. Therefore $$2^{\frac{r-1}{r}k^2} \leq P((k+3r)^2,k+3r+1)$$
for any fixed $r$. For fixed $r$ and $k$ and $n$ tending to infinity a product construction implies $$2^{\frac{r-1}{r}n-o(n)} \leq  P(n,k+3r+1). $$
Since we can choose $r$ arbitrarily large, we have $$2 \leq \liminf_{k \rightarrow \infty}P(k)$$ as claimed.

\end{proof}
Theorem \ref{thm:asymlower} and Corollary \ref{corr:upper} together implies \[\lim_{k \rightarrow \infty} P(k)=2.\]

\section{The behaviour of $P(n,n)$} \label{sec:pnn}

For our lower bound on $P(n,n)$ we will use the following lemma.

\begin{lem} \label{three}
When $n \geq 4$ is even, there are three perfect matchings $M_1,M_2,M_3$ on $n$ vertices such that all three pairwise unions form a Hamiltonian cycle.
\end{lem}
\begin{proof}
Let

\begin{align*}
M_1 &:= \{(1,2), (3,4) , \ldots, (n-1,n)\} \\
M_2 &:= \{(n,1), (2,3) , \ldots, (n-2,n-1)\}.
\end{align*}

Note that $M_1 \cup M_2$ is a Hamiltonian cycle. If $n=4k+2$ it is easy to see that the matching
$$M_3 := \{(1,n/2+1),(2,n/2+2), \ldots, (n/2,n) \} $$
completes both $M_1$ and $M_2$ to a Hamiltonian cycle. If $n=4k$ then consider the following matching that consists of the shifts of the edges $\{(1,4),(3,6)\}$:

$$M_3':= \{(1,4),(3,6)  , (5,8),(7,10) , \ldots (1+4l,4+4l),(3+4l,6+4l), \ldots , (n-4,n),(n-2,2) \}. $$
($M_1 \cup M_3'$ is a Hamiltonian cycle but $M_2 \cup M_3'$ is the union of two cycles of length $n/2$.) We define $M_3$ by changing only two edges in $M_3'$:
$$M_3 := (M_3' \setminus \{(1,4),(3,6)\} ) \cup \{(1,3),(4,6)\}.$$
It is again easy to see that $M_3$ completes both $M_1$ and $M_2$ to a Hamiltonian cycle and the proof is complete.
\end{proof}

\begin{remark}
It is conjectured by several authors that for all $n$, the maximal number of perfect matchings on $n$ vertices where every pairwise union is a Hamiltonian cycle is $n-1$, see the survey  \cite{survey}. If such a set of perfect matchings exists, it is called a \textit{perfect $1$-factorization of the complete graph}.
\end{remark}

\begin{proof}[Proof of Theorem \ref{k=n}]
We define a labelled $H$-cycle to be a labelled graph that is a Hamiltonian cycle with a single edge that gets label $b$ and every other edge gets label $a$. Observe that $P(n,n)$ is the maximal number of pairwise compatible labelled $H$-cycles.

\textbf{For the upper bound:} Suppose that we have a family $\mathcal{H}$ of pairwise compatible labelled $H$-cycles on the vertex set $[n]$. We define a new graph $G$ on the vertices $[n]$. Two vertices $x$ and $y$ in $G$ are adjacent if there is a labelled $H$-cycle in $\mathcal{H}$ that contains the edge $\{x,y\}$ with label $b$. Since every labelled $H$-cycle contains only one edge of label $b$ and $\mathcal{H}$ consists of pairwise compatible labelled graphs we have  $$ E(G)=|\mathcal{H}|. $$

\begin{claim}
 $\Delta(G) \leq 3$
\end{claim}

\begin{proof}
Suppose that we have a vertex in $G$ that has degree at least $4$. Choose four edges that are incident to this vertex. These edges correspond to pairwise compatible labelled $H$-cycles. The edges that get label $b$ in these labelled $H$-cycles form a star. Therefore for each such $H$-cycle $H_1$, there is at most one other among the other three, that has its edge of label $b$ contained in $H_1$ with label $a$. Hence there are at most four compatible pairs among these four labelled $H$-cycles. But for pairwise compatibility there must be at least six, a contradiction.
\end{proof}

Since $\Delta(G) \leq 3$, by the sum of the degrees in $G$, we have $|E(G)| \leq \left \lfloor \frac{3}{2}n \right \rfloor$ and the proof of the upper bound is complete.

\textbf{For the lower bound: } If $n$ is even, by Lemma \ref{three} there are three perfect matchings $M_1, M_2, M_3$ on $[n]$ such that every pairwise union of these is a Hamiltonian cycle. We define three sets of labelled $H$-cycles as follows (all the indices in the following definition are understood modulo $3$). For $i=1,2,3$ let  $E_i $  be a set of labelled $H$-cycles that contain every edge of $M_i$ and $M_{i+1}$. One edge of $M_i$ gets label $b$ and every other edge gets label $a$. Let

$$\mathcal{F}= E_1 \cup E_2 \cup E_3.$$

Clearly $|E_i|=n/2$ for all $i$ thus $ |\mathcal{F}|= \frac{3}{2}n$. It is easy to see that $\mathcal{F}$ consists of pairwise compatible labelled $H$-cycles. 
If $n$ is odd, we proceed similarly as in the even case. We take the three perfect matchings $M_1,M_2$ and $M_3$ on $n-1$ vertices, guaranteed by Lemma \ref{three}. But the pairwise unions of these matchings are not Hamiltonian cycles, but cycles of length $n-1$. We change the matchings  by replacing a single edge in $M_1$ and $M_2$ with an edge that is incident to the $n$-th vertex in such a way that no two matchings share an isolated vertex. After this change we proceed in the same way as in the even case. The only difference is that now we only have $|\mathcal{F}|=\frac{3(n-1)}{2}=\left \lfloor  \frac{3n}{2}\right \rfloor -1 .$

\end{proof}

\begin{remark}
It is not hard to show that when $n$ is odd, there are no three graphs $M_1,M_2,M_3$ with the following properties 
\begin{itemize}
\item $M_1,M_2,M_3$ are edge disjoint
\item Each pairwise union from $\{M_1,M_2,M_3\}$ is a subgraph of a Hamiltonian cycle
\item $|E(M_1)|+|E(M_2)|+|E(M_3|= \left \lfloor  \frac{3n}{2}\right \rfloor. $
\end{itemize}
Therefore the strategy that we used to show that for even $n$, $\frac{3n}{2} \leq P(n,n)$ can not be used when $n$ is odd.  
\end{remark}

\section{Open problems and concluding remarks} \label{opr}

In \cite{original} the authors determined the exact value of $P(n,3)$ and its rough order of magnitude is $2^{n-o(n)}$. The main results of the present paper suggest that the order of magnitude of $P(n,k)$ is also $2^{n-o(n)}$ when $k$ is odd.

In the present paper the methods that lead to the exact value of $P(n,3)$ were applied to the case of $P(n,2^\ell +1)$ and in these special cases we managed to prove Conjecture \ref{conj:main}. Our results in Section \ref{sec:otherk} imply that  $(2-\varepsilon(k))^n \leq P(n,k) \leq (2+\varepsilon(k))^n $ where $\varepsilon(k)$ tends to zero with $k$ tending to infinity. Therefore Conjecture \ref{conj:main} holds asymptotically.

When proving Conjecture \ref{conj:main} for $k=2^\ell +1$, we constructed pairwise compatible labelled graphs with a linear number of labelled grids that got us $2^{n-o(n)}$ permutations. It would be useful to construct a family of $2^{n-o(n)}$ pairwise compatible labelled graphs of any fixed width $k$ with a single labelled grid. If we can do this, it not only implies $P(n,k+1) \approx 2^{n-o(n)}$, but also $P(n,k \ell+1) \approx 2^{n-o(n)}$ for every positive integer $\ell$, as follows. We can divide $n$ vertices into $\ell$ equal parts and take a single-grid construction on each part and take the ``product'' of these constructions. This way our labelled graphs contain exactly $\ell$ labelled grids and we can merge these as the $W_1$ part of Figure \ref{fig:merge} to get a single grid of width $k \ell$. This way we got a set of  $2^{n-o(n)}$ compatible labelled graphs of width $k \ell$. Note that just merging the grids  as the $W_1$ part of Figure \ref{fig:merge} decreases the exponent if we have a linear number of labelled graphs.

Observe that two labelled grids are compatible if and only if their ``transposed'' versions are compatible. The transposed version of a grid with width $2$ and height $n/2$ is one with width $n/2$ and height $2$. The maximal number of pairwise compatible labelled grids of width $n/2$ and height $2$ is just a polynomial factor away from $P(n,n/2+1)$ (it is trivially a lower bound, and for the upper bound consider the middle edge of each permutation and take only those that have the most frequent middle edge, this way we only lose a factor of $\binom{n}{2}$). Therefore for any fixed $k$ we have $$\lim_{n \rightarrow \infty}\sqrt[n]{P(n,n/2+1)} \leq \lim_{n \rightarrow \infty}\sqrt[n]{P(n,2k+1)}.$$
Since $P(n,3) \leq 2^n$, we have $ \lim_{n \rightarrow \infty}\sqrt[n]{P(n,n/2+1)}\leq 2 $.

\begin{question}
What is the limit of $ \sqrt[n]{P(n,n/2+1)}$ when $n$ tends to infinity?
\end{question}

\begin{defi}
The graph $G_3$ is the union of two graphs $G_1,G_2$ on the same vertex set if $V(G_3)=V(G_2)=V(G_1)$ and $E(G_3)= E(G_1) \cup E(G_2)$. We say that two Hamiltonian paths $H_1,H_2$ on the same vertex set $[n]$ are $G$-different if $G$ is a (not necessarily induced) subgraph of $H_1 \cup H_2$. Let $H(n,k)$ be the maximal number of $C_k$-different Hamiltonian paths on $n$ vertices.
\end{defi}

Observe that $H(n,3)=P(n,3)$. If we wish to generalize Theorem \ref{previousmain} in the language of Hamiltonian paths, its natural generalization is to determine $H(n,k)$ for fixed $k$. For odd $k$, our best lower bounds for $H(n,k)$ come from the inequality $P(n,k) \leq H(n,k)$ and the results of the present paper. For odd $k$, the upper bound $H(n,k)\leq 2^{n-o(n)}$ holds by the argument used in Proposition \ref{prop}.

For even $k$, $H(n,k)$ is much larger than $P(n,k)$. In \cite{k=4} it is proven that $n^{\frac{1}{2}n-o(n)} \leq H(n,4)\leq n^{\frac{3}{4}n-o(n)}$. The lower bound is easy: The family of directed Hamiltonian paths of $[n]$ where for every $i$, the $(2i+1)$-th vertex of every Hamiltonian path is the vertex $(2i+1) \in [n]$ satisfies the conditions. It is not hard to generalize this construction to yield $n^{\frac{1}{k}n-o(n)} \leq H(n,2k)$. The upper bound in \cite{k=4} is much more involved, and in a subsequent paper it will be generalized to yield $H(n,k) \leq n^{\left( 1- \frac{1}{ck^2} \right)n-o(n)}$. Therefore our knowledge of $H(n,2k)$ can be summarized as follows
$$n^{\frac{1}{k}n-o(n)} \leq H(n,2k) \leq n^{\left( 1- \frac{1}{ck^2} \right)n-o(n)}. $$
For $P(n,k)$ and $H(n,2k+1)$, if $k$ is getting larger, our bounds are improving, but in the case of $H(n,2k)$ they are getting worse.

\section{Acknowledgement}

We would like to thank Gábor Simonyi and Géza Tóth for their help which improved the quality of this manuscript.

\end{document}